\documentclass[reqno]{amsart}
\newtheorem{theorem}{Theorem}[section]
\newtheorem{lemma}[theorem]{Lemma}
\newtheorem{proposition}[theorem]{Proposition}
\newtheorem{corollary}[theorem]{Corollary}
\theoremstyle{definition}
\newtheorem{definition}[theorem]{Definition}
\theoremstyle{plain}

\usepackage{graphicx}
\usepackage{amscd}
\usepackage{amsmath, amssymb}
\usepackage{graphics}
\usepackage{caption}
\usepackage{subcaption}
\usepackage{epsfig}
\usepackage{xcolor}
\usepackage{soul}
\usepackage{tikz}

\newtheorem{question}[theorem]{Question}

\numberwithin{equation}{section}

\begin{document}

\title[Representations of the twisted virtual braid group]{Matrix representations of the twisted virtual braid group and its extensions}

\author{Mohamad N. Nasser}
\address{(Mohamad N. Nasser) Department of Mathematics and Computer Science, Beirut Arab University, Beirut, Lebanon}
\email{m.nasser@bau.edu.lb}

\author{Vaibhav Keshari}
\address{(Vaibhav Keshari) Department of Mathematics, Indian Institute of Technology Ropar, Punjab, India}
\email{vaibhav.23maz0022@iitrpr.ac.in}

\author{Madeti Prabhakar}
\address{(Madeti Prabhakar) Department of Mathematics, Indian Institute of Technology Ropar, Punjab, India}
\email{prabhakar@iitrpr.ac.in}

\maketitle

\begin{abstract}
This paper classifies complex local representations of the twisted virtual braid group, $TVB_2$, into $\mathrm{GL}_3(\mathbb{C})$. It shows that such representations fall into eight types, all of which are unfaithful and reducible to a two-dimensional representation. Further reducibility to a one-dimensional representation is analyzed for specific types. The paper also examines complex homogeneous local representations of $TVB_n$ into $\mathrm{GL}_{n+1}(\mathbb{C})$ for $n \geq 3$, identifying seven unfaithful types. Additionally, complex local representations of the singular twisted virtual braid monoid, $STVB_2$, into $\mathrm{M}_3(\mathbb{C})$ are classified into thirteen unfaithful types. Finally, the paper demonstrates that not all complex local extensions of $TVB_2$ representations to $STVB_2$ conform to a $\Phi$-type extension.
\end{abstract}

\renewcommand{\thefootnote}{}

\section{Introduction}

The braid group $B_n$ on $n$ strands was introduced by Emil Artin in 1925~\cite{artin1946}. It is generated by $\sigma_1, \sigma_2, \ldots, \sigma_{n-1}$, subject to relations modelling the topological behaviour of braiding, and plays a central role in low-dimensional topology, algebra, and mathematical physics. In 2004, Kauffman and Lambropoulou~\cite{KauffmanLambropoulou2004} introduced the virtual braid group $VB_n$, which generalizes $B_n$ by adding \emph{virtual crossings} and forms the algebraic foundation of \emph{virtual knot theory}. In 2024, Negi, Prabhakar, and Kamada~\cite{NegiPrabhakarKamada2023} introduced the twisted virtual braid group $TVB_n$, which extends $VB_n$ by adding generators and relations that model \emph{twists} in virtual braids, allowing it to capture further knot-theoretic phenomena.

Earlier, Baez~\cite{Baez1992} in 1992 and Birman~\cite{Birman1993} in 1993 independently introduced the singular braid monoid $SM_n$, which generalizes the braid group by allowing \emph{singular crossings}, interpreted as collisions of strands, and is also given by generators and relations. Most recently, in 2025, Prabhakar and Negi~\cite{Pra2025} introduced the singular twisted virtual braid monoid $STVB_n$, which combines singularity and twisting for virtual braids. They also defined a group $STVG_n$ into which $STVB_n$ embeds. Together, these structures provide a rich framework for studying braids in virtual and singular settings.

The Lawrence--Krammer--Bigelow representation of $B_n$ is faithful for all $n$, so the braid group is linear. Mikhalchishina~\cite{Mik2013} classified the \emph{local representations} of $B_3$ and the \emph{homogeneous 2-local representations} of $B_n$ for $n \geq 3$. Mayassi and Nasser~\cite{May2025} extended these results by classifying all \emph{homogeneous 3-local representations} of $B_n$ for $n \geq 4$.

In~\cite{Bar2024}, Bardakov, Chbili, and Kozlovskaya introduced \emph{$\Phi$-type extensions} of representations of $B_n$ to $SM_n$, and Nasser~\cite{Nas2024} studied their faithfulness in some cases. Prabhakar and Negi~\cite{Pra2025} carried this idea over to the twisted virtual setting, from $TVB_n$ to $STVB_n$. The Birman representation, proved faithful by Paris~\cite{Paris2001}, is a special case of a $\Phi$-type extension from $B_n$ to $SM_n$, which shows that $SM_n$ is linear. Deciding whether $TVB_n$, $STVB_n$, and $STVG_n$ are linear requires constructing and studying their representations. In this paper, we study complex homogeneous local representations of these structures and their properties, such as faithfulness and irreducibility.\vspace{0.1cm}

We begin in Section 2 by presenting the main definitions and propositions needed throughout the paper. In Section 3, we study all complex local representations of $TVB_2$ and determine whether they are faithful and irreducible. This section also introduces complex homogeneous local representations of $TVB_n$ for $n \geq 3$ and analyzes their key properties. Section 4 focuses on complex local representations of $STVB_2$, evaluating their faithfulness and irreducibility and posing several open questions for future work. Finally, in Section 5, we determine $\Phi$-type extensions of complex local representations of $TVB_2$ to $STVB_2$ and discuss their structural and representational implications.

\section{Main Definitions}

The braid group $B_n$~\cite{artin1946}, $n\geq2$, on $n$ strands can be defined as a group generated by $\sigma_1$, $\sigma_2$,$\dots$, $\sigma_{n-1} $ with the defining relations 
\begin{align}
\sigma_i\sigma_{i+1}\sigma_i &= \sigma_{i+1}\sigma_i\sigma_{i+1}, 
& i = 1,\ldots,n-2,\\
\sigma_i\sigma_j &= \sigma_j\sigma_i, 
& |i-j| \ge 2.
\end{align}
\begin{figure}[h]
\centering
    \begin{subfigure}[c]{0.4\textwidth}
        \centering
     \includegraphics[width=0.6\textwidth]{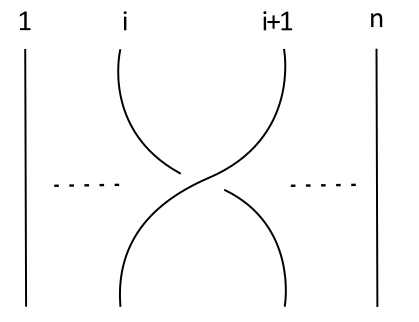}
     \caption{$\sigma_i$}
     \label{fig:cl1}
 \end{subfigure}
 ~
\begin{subfigure}[c]{0.4\textwidth}
            \centering
            \includegraphics[width=0.6\textwidth]{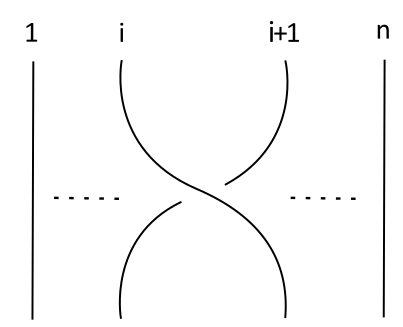}
         \caption{$\sigma^{-1}_i$} 
    \end{subfigure}
    \caption{~Geometrical interpretation of $\sigma_i$ and $\sigma^{-1}_i$.}
    \label{fig:cl2}
\end{figure}

The virtual braid group $VB_n$ defined in \cite{KauffmanLambropoulou2004} is an extension of the classical braid group $B_n$ generated by the generators of $B_n$ and new generators, \( \rho_1, \rho_2, \dots, \rho_{n-1},\) which satisfy the  relations (2.1)-(2.2) along with the following new relations:
\begin{align}
\rho_i^2 &= e, 
& i = 1,2,\ldots,n-1,\\
\rho_i \rho_j &= \rho_j \rho_i, 
& |i-j| \geq 2,\\
\rho_i \rho_{i+1} \rho_i &= \rho_{i+1} \rho_i \rho_{i+1}, 
& i = 1,2,\ldots,n-2,\\
\sigma_i \rho_j &= \rho_j \sigma_i, 
& |i-j| \geq 2,\\
\rho_i \rho_{i+1}\sigma_i &= \sigma_{i+1}\rho_i\rho_{i+1}, 
& i = 1,2,\ldots,n-2.
\end{align}
\begin{figure}[h]
\centering
    \begin{subfigure}[c]{0.4 \textwidth}
        \centering
     \includegraphics[width=0.6\textwidth]{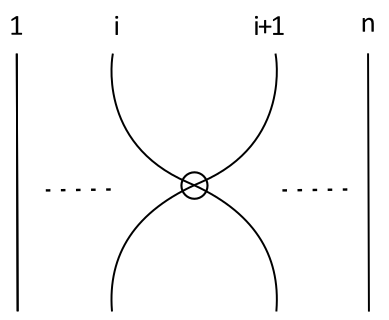}
     \caption{$\rho_i$}
     \label{fig:vir}
 \end{subfigure}
 ~
\begin{subfigure}[c]{0.4\textwidth}
            \centering
            \includegraphics[width=0.5\textwidth]{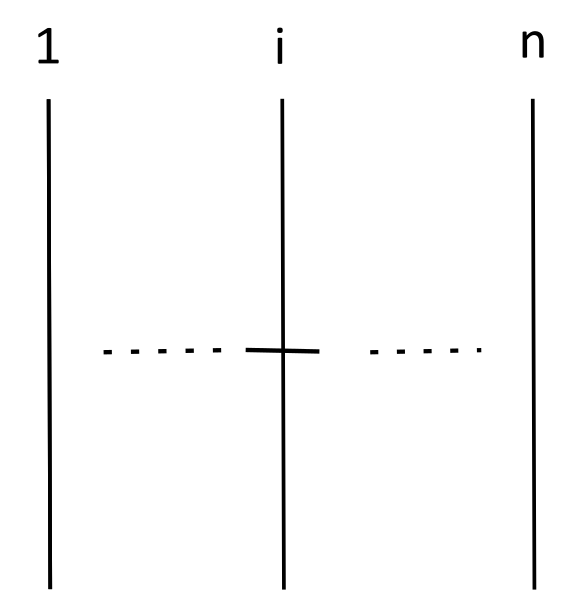}
         \caption{$\gamma_i$} 
         
    \end{subfigure}
    \vspace*{8pt}
    \caption{~Geometrical interpretation of $\rho_i$ and $\gamma_i$.}
    \label{virtual and twisted}
\end{figure}

The twisted virtual braid group $TVB_n$ defined in \cite{NegiPrabhakarKamada2023} is a group generated by the generators of $VB_n$ and new generators, $\gamma_1, \gamma_2, \dots, \gamma_n$, which satisfy the relations of $VB_n$ (2.1)-(2.7) along with the relations:
\begin{align}
\gamma_i^2 &= e, 
& i = 1,2,\ldots,n,\\
\gamma_i \gamma_j &= \gamma_j \gamma_i, 
& i,j = 1,2,\ldots,n,\\
\gamma_j \rho_i &= \rho_i \gamma_j, 
& |i-j| \geq 2,\\
\gamma_j \sigma_i &= \sigma_i \gamma_j, 
& |i-j| \geq 2,\\
\rho_i \gamma_i &= \gamma_{i+1} \rho_i, 
& i = 1,2,\ldots,n-1,\\
\rho_i \sigma_i \rho_i &= \gamma_{i+1} \gamma_i \sigma_i \gamma_i \gamma_{i+1}, 
& i = 1,2,\ldots,n-1.
\end{align}

The singular braid monoid $SM_n$ defined in \cite{Baez1992,Birman1993} is a monoid generated by the elements $\sigma_i, \sigma^{-1}_i, \tau_i,~ i=1,2,\dots, n-1$, where the elements $\sigma_i$~and$~\sigma^{-1}_i$ are the generators of the braid group $B_n$ satisfying the relations (2.1)-(2.2) and the generators $\tau_i$ satisfy the defining relations:
\begin{align}
\tau_i\tau_j &= \tau_j\tau_i, 
& |i-j| \geq 2,
\end{align}
and the mixed relations
\begin{align}
\tau_i\sigma_j &= \sigma_j\tau_i, 
& |i-j| \geq 2,\\
\tau_i\sigma_i &= \sigma_i\tau_i, 
& i = 1,2,\ldots,n-1,\\
\sigma_i\sigma_{i+1}\tau_i &= \tau_{i+1}\sigma_i\sigma_{i+1}, 
& i = 1,2,\ldots,n-2,\\
\sigma_{i+1}\sigma_i\tau_{i+1} &= \tau_i\sigma_{i+1}\sigma_i, 
& i = 1,2,\ldots,n-2.
\end{align}
The singular braid group $SB_n$ defined in \cite{Roger1998} has the same generators as $SM_n$ with an additional generator $\bar{\tau_i}, 1\leq i<n,$ satisfying the defining relations: 
    \begin{enumerate}
      \item  The same relations as in $SM_n$, i.e., (2.1)-(2.2) and (2.14)-(2.18), along with the additional relations obtained by substituting $\bar{\tau_i}$ for $\tau_i$ in each relevant relation of $SM_n$, i.e.,(2.14)-(2.18),
      \item $\tau_i\bar{\tau_i}=\bar{\tau_i}\tau_i=e, \quad i=1,2,\ldots, n-1.$
    \end{enumerate} 
\begin{figure}[h]
\centering
    \begin{subfigure}[c]{0.4\textwidth}
        \centering
     \includegraphics[width=0.6\textwidth]{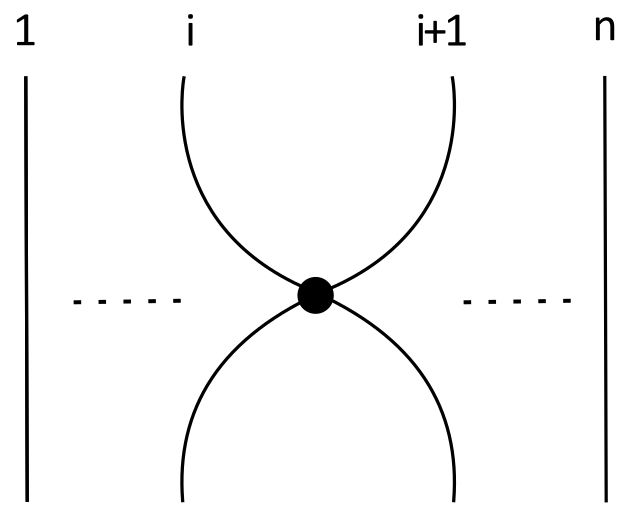}
     \caption{$\tau_i$}
     \label{fig:sing}
 \end{subfigure}
 ~
\begin{subfigure}[c]{0.4\textwidth}
            \centering
            \includegraphics[width=0.6\textwidth]{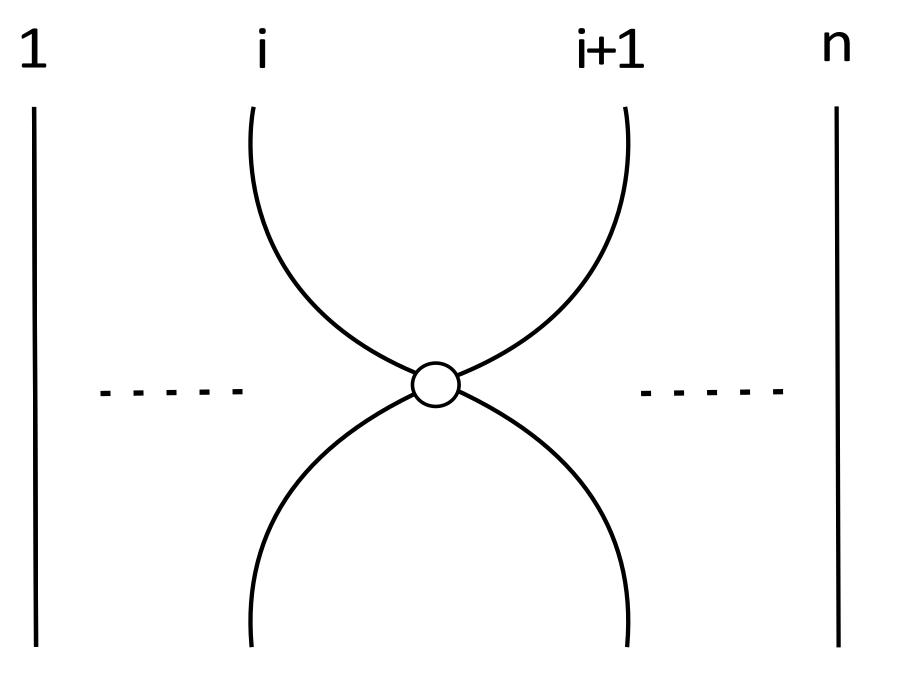}
         \caption{$\bar{\tau_i}$} 
    \end{subfigure}
    \caption{~~Geometrical interpretation of $\tau_i$ and $\bar{\tau_i}$.}
\end{figure}
   
The singular twisted virtual braid monoid $STVB_n$ is generated by standard generators, $\sigma^{\pm 1}_i, \rho_i, \tau_i ~ (i=1,2,\ldots,n-1)$ and $\gamma_i ~(i=1,2,\ldots,n)$ with the defining relations (2.1)-(2.18) and 
\begin{align}
\tau_i\rho_j &= \rho_j\tau_i, 
& |i-j|\geq 2,\\
\rho_i\tau_{i+1}\rho_i &= \rho_{i+1}\tau_i\rho_{i+1}, 
& i = 1,2,\ldots,n-2,\\
\tau_i\gamma_j &= \gamma_j\tau_i, 
& |i-j|\geq 2,\\
\rho_i\tau_i\rho_i &= \gamma_{i+1}\gamma_i\tau_i\gamma_i\gamma_{i+1}, 
& i = 1,2,\ldots,n-1.
\end{align}
  
The singular twisted virtual braid group, $STVG_n$, defined in \cite{Pra2025} is a group generated by the same generators as $STVB_n$, together with $\bar{\tau_i}, 1\leq i<n,$ satisfying the defining relations:
  \begin{enumerate}
  \item  The same relations as in $STVB_n$, i.e.,(2.1)-(2.22), along with the additional relations obtained by substituting $\bar{\tau_i}$ for $\tau_i$ in each relevant relation of $STVB_n$, i.e.,(2.14)-(2.22),
  \item $\tau_i\bar{\tau_i}=\bar{\tau_i}\tau_i=e, \quad i=1,2,\ldots, n-1,$
  \item $\tau_j\bar{\tau_i}=\bar{\tau_i}\tau_j, \quad |i-j|\geq 2.$
  \end{enumerate}

The monoid $SM_n$ embeds in $SB_n$, and $STVB_n$ embeds in $STVG_n$. Studying the linearity of these groups, i.e.\ whether they admit faithful representations into general linear groups, requires understanding their representations. We first define group representations and their basic properties, then discuss the types relevant to our study, and finally present the Burau representation of $B_n$ and its key results.

\begin{definition}\cite{Vinberg} Let \( G \) be a group. A representation of \( G \) of degree \( n \) is a group homomorphism
\[
\lambda : G \longrightarrow \mathrm{GL}_n(\mathbb{C}),
\]
where \( \mathrm{GL}_n(\mathbb{C}) \) is the group of all invertible \( n \times n \) matrices with complex entries. 
\end{definition}

\begin{definition}\cite{Vinberg} A representation \( \lambda : G \longrightarrow \mathrm{GL}_n(\mathbb{C}) \) is said to be irreducible if there is no nontrivial proper subspace \( W \subset \mathbb{C}^n \) such that \( \lambda(g)(w) \in W \) for all \( g \in G \) and \( w \in W \). 
\end{definition}

\begin{definition}\cite{Vinberg} A representation \( \lambda : G \to \mathrm{GL}_n(\mathbb{C}) \) is said to be faithful if it is injective; that is, \( \ker(\lambda) = \{e\} \), where \( e \) is the identity element of \( G \).
\end{definition}

\begin{definition}\cite{Bur1936} \label{defBurau}
For $t$ indeterminate, the Burau representation $\mathcal{B}: B_n\longrightarrow \mathrm{GL}_n(\mathbb{Z}[t^{\pm 1}])$, for $n\geq 2$, is the representation given as follows.
$$\mathcal{B}(\sigma_i)= \left( \begin{array}{c|@{}c|c@{}}
   \begin{matrix}
     I_{i-1} 
   \end{matrix} 
      & \textbf{0} & \textbf{0} \\
      \hline
    \textbf{0} &\hspace{0.2cm} \begin{matrix}
   	1-t \ & t \ \\
   	1 \ & 0 \ \\
\end{matrix}  & \textbf{0}  \\
\hline
\textbf{0} & \textbf{0} & I_{n-i-1}
\end{array} \right) \hspace*{0.2cm} \text{for} \hspace*{0.2cm} 1\leq i\leq n-1.$$ 
\end{definition}
The faithfulness of the Burau representation depends on $n$: it is faithful for $n\leq 3$ \cite{Bir1975}, unfaithful for $n\geq 5$ (\cite{Moo1991}, \cite{Long1992}, \cite{Big1999}), and the case $n=4$ remains open. As for irreducibility, it is reducible for $n\geq 2$ \cite{For1996}.

Now, we introduce the definition and main results of another famous representation of $B_n$, which is the Lawrence-Krammer-Bigelow (LKB) representation.

\begin{definition} \cite{Law1990,Kram2002,Big2001}. \label{defLaw}
Let $R=\mathbb{Z}[t^{\pm 1},q^{\pm 1}]$ be the ring of Laurent polynomials in $q$ and $t$, and let $V$ be the free $R$-module with basis $\{x_{i,j} : 1 \leq i < j \leq n\}$. For $n\geq 3$, the LKB representation $\mathcal{K}: B_n\longrightarrow \mathrm{GL}_{\frac{n(n-1)}{2}}(R)$ acts on the generators $\sigma_k$, $1\leq k \leq n-1$, as follows.
\[
\mathcal{K}(\sigma_k)(x_{i,j}) =
\begin{cases}
tq^2 x_{k,k+1}, & i=k<k+1=j,\\[4pt]
(1-q)x_{i,k} + qx_{i,k+1}, & i<k=j,\\[4pt]
x_{i,k} + tq^{k-i+1}(q-1)x_{k,k+1}, & i<k<k+1=j,\\[4pt]
tq(q-1)x_{k,k+1} + qx_{k+1,j}, & i=k<k+1<j,\\[4pt]
x_{k,j} + (1-q)x_{k+1,j}, & k<i=k+1<j,\\[4pt]
x_{i,j}, & i<j<k \ \text{or}\ k+1<i<j,\\[4pt]
x_{i,j} + tq^{k-i}(q-1)^2 x_{k,k+1}, & i<k<k+1<j.
\end{cases}
\]
\end{definition}
For the faithfulness of the LKB representation, it was shown in \cite{Big2001} and \cite{Kram2002} that it is faithful for all $n\geq 3$. Regarding irreducibility, the LKB representation is shown to be irreducible for all $n\geq 3$ in \cite{Lev2010}.\\

In what follows, we give the concept of local representations of any group $G$ with finite number of generators.

\begin{definition}
Let $t$ be indeterminate and let $G$ be a group with generators $g_1,g_2,\ldots,g_{n-1}$. A representation $\theta: G \longrightarrow \mathrm{GL}_{m}(\mathbb{Z}[t^{\pm 1}])$ is said to be local if it is of the form
$$\theta(g_i) =\left( \begin{array}{c|@{}c|c@{}}
   \begin{matrix}
     I_{i-1} 
   \end{matrix} 
      & \textbf{0} & \textbf{0} \\
      \hline
    \textbf{0} &\hspace{0.2cm} \begin{matrix}
   		M_i \ 
   		\end{matrix}  & \textbf{0}  \\
\hline
\textbf{0} & \textbf{0} & I_{n-i-1}
\end{array} \right) \hspace*{0.2cm} \text{for} \hspace*{0.2cm} 1\leq i\leq n-1,$$ 
where $M_i \in \mathrm{GL}_k(\mathbb{Z}[t^{\pm 1}])$ with $k=m-n+2$ and $I_r$ is the $r\times r$ identity matrix. The local representation is said to be homogeneous if all the matrices \(M_i, 1 \leq i \leq n-1\), are equal.
\end{definition}

We can see that the Burau representation defined in Definition \ref{defBurau} is a homogeneous local representation, while LKB representation defined in Definition \ref{defLaw} is not local.\\

We remark that the concept of local representations could be extended if $G$ is a group of $r(n-1)$ generators, where the generators of $G$ consist of $r$ families and each family consists of $n-1$ members. 
The definition below is given for \( r = 2 \).
 
\begin{definition}
Let $t$ be indeterminate and let $G$ be a group with $2$ families of generators $g_1,g_2,\ldots,g_{n-1}$ and $h_1,h_2,\ldots,h_{n-1}$. A local representation $\theta: G \longrightarrow \mathrm{GL}_{m}(\mathbb{Z}[t^{\pm 1}])$ is a representation of the form
$$\theta(g_i) =\left( \begin{array}{c|@{}c|c@{}}
   \begin{matrix}
     I_{i-1} 
   \end{matrix} 
      & \textbf{0} & \textbf{0} \\
      \hline
    \textbf{0} &\hspace{0.2cm} \begin{matrix}
   		M_i \ 
   		\end{matrix}  & \textbf{0}  \\
\hline
\textbf{0} & \textbf{0} & I_{n-i-1}
\end{array} \right) \text{ and \ } \theta(h_i) =\left( \begin{array}{c|@{}c|c@{}}
   \begin{matrix}
     I_{i-1} 
   \end{matrix} 
      & \textbf{0} & \textbf{0} \\
      \hline
    \textbf{0} &\hspace{0.2cm} \begin{matrix}
   		N_i \ 
   		\end{matrix}  & \textbf{0}  \\
\hline
\textbf{0} & \textbf{0} & I_{n-i-1}
\end{array} \right) $$
for $1\leq i\leq n-1,$ where $M_i,N_i \in \mathrm{GL}_k(\mathbb{Z}[t^{\pm 1}])$ with $k=m-n+2$ and $I_r$ is the $r\times r$ identity matrix. In this case, $\theta$ is homogeneous if all the matrices $M_i$'s are equal and all the matrices $N_i$'s are equal.
\end{definition}

In the following definition, we aim to extend the concept of local representations to $TVB_n$. As we mentioned previously, $TVB_n$ has three families of generators: $\sigma_1,\sigma_2,\ldots, \sigma_{n-1}$, $\rho_1,\rho_2,\ldots, \rho_{n-1}$, and  $\gamma_1,\gamma_2,\ldots, \gamma_{n}$. It is evident that the number of generators in the third family exceeds the number of generators in each of the first two families by $1$. Therefore, we introduce the concept of local representations for $TVB_n$ as follows.

\begin{definition}
Let $t$ be indeterminate. A local representation $\theta: TVB_n \longrightarrow \mathrm{GL}_{m}(\mathbb{Z}[t^{\pm 1}])$ is a representation that acts on the three families of generators of $TVB_n$ as follows.
$$\theta(\sigma_i) =\left( \begin{array}{c|@{}c|c@{}}
   \begin{matrix}
     I_{i-1} 
   \end{matrix} 
      & \textbf{0} & \textbf{0} \\
      \hline
    \textbf{0} &\hspace{0.2cm} \begin{matrix}
   		S_i \ 
   		\end{matrix}  & \textbf{0}  \\
\hline
\textbf{0} & \textbf{0} & I_{n-i}
\end{array} \right),$$ 
$$\theta(\rho_i) =\left( \begin{array}{c|@{}c|c@{}}
   \begin{matrix}
     I_{i-1} 
   \end{matrix} 
      & \textbf{0} & \textbf{0} \\
      \hline
    \textbf{0} &\hspace{0.2cm} \begin{matrix}
   		R_i \ 
   		\end{matrix}  & \textbf{0}  \\
\hline
\textbf{0} & \textbf{0} & I_{n-i}
\end{array} \right),$$ 
and
$$\theta(\gamma_j) =\left( \begin{array}{c|@{}c|c@{}}
   \begin{matrix}
     I_{j-1} 
   \end{matrix} 
      & \textbf{0} & \textbf{0} \\
      \hline
    \textbf{0} &\hspace{0.2cm} \begin{matrix}
   		G_j \ 
   		\end{matrix}  & \textbf{0}  \\
\hline
\textbf{0} & \textbf{0} & I_{n-j}
\end{array} \right)$$
for $1\leq i\leq n-1$ and  $1\leq j\leq n,$ where $S_i,R_i,G_j \in \mathrm{GL}_k(\mathbb{Z}[t^{\pm 1}])$ with $k=m-n+1$ and $I_r$ is the $r\times r$ identity matrix. In this case, $\theta$ is homogeneous if all the matrices $S_i$'s are equal, all the matrices $R_i$'s are equal, and all the matrices $G_j$'s are equal.
\end{definition}

In the following proposition, we introduce the concept of $\Phi$-type extensions of representations of the braid group $B_n$ to the singular braid monoid $SM_n$, given by Bardakov, Chbili, and Kozlovskaya in \cite{Bar2024}. 
 
\begin{proposition} \label{Phi}
Let $\phi: B_n \longrightarrow G_n$ be a representation of $B_n$ into a group $G_n$, let $\mathbb{K}$ be a field, and let $a,b,c \in \mathbb{K}$. Then the map $\Phi_{a,b,c}:SM_n\longrightarrow \mathbb{K}[G_n]$ defined on generators, for $i=1,2,\ldots,n-1$, by
$$\Phi_{a,b,c}(\sigma_i^{\pm 1})=\phi(\sigma_i^{\pm 1}) \ \text{ and }\  \Phi_{a,b,c}(\tau_i)=a\phi(\sigma_i)+b\phi(\sigma_i^{-1})+ce,$$
where $e$ is the identity of $G_n$, is a representation of $SM_n$ into $\mathbb{K}[G_n]$.
\end{proposition}

In \cite{Pra2025}, Prabhakar and Negi extended the concept of $\Phi$-type extensions given in Proposition \ref{Phi} to $TVB_n$. They proved, in a similar way, that if $\phi: TVB_n \longrightarrow G_n$ is a representation of $TVB_n$ to a group $G_n$ and $\mathbb{K}$ is a field with $a,b,c \in \mathbb{K}$, then, the map $\Phi_{a,b,c}: STVB_n\longrightarrow \mathbb{K}[G_n]$ which acts on the generators of $STVB_n$ by the rules
$$\Phi_{a,b,c}(\sigma_i^{\pm 1})=\phi(\sigma_i^{\pm 1}),\ \Phi_{a,b,c}(\rho_i)=\phi(\rho_i),\ \Phi_{a,b,c}(\gamma_j)=\phi(\gamma_j) \ \text{ and }$$ $$\Phi_{a,b,c}(\tau_i)=a\phi(\sigma_i)+b\phi(\sigma_i^{-1})+ce,\ i=1,2,\ldots,n-1, j=1,2,\ldots, n,$$  where $e$ is the neutral element of $G_n$, defines a representation of $STVB_n$ to $\mathbb{K}[G_n]$.

\section{On Complex Local Representations of $TVB_n$}

In this section, we determine the forms of all complex homogeneous local representations of $TVB_n$ for all $n\geq 2$. In addition, we study the irreducibility and the faithfulness for these representations. In the following theorem, we begin with \( n = 2 \).
\begin{theorem} \label{localTVB2}
Let $\zeta: TVB_2 \longrightarrow \mathrm{GL}_3(\mathbb{C})$ be a complex local representation of $TVB_2$. Then, $\zeta$ is equivalent to one of the following eight representations.
\begin{itemize}
\item[(1)] $\zeta_1: TVB_2 \longrightarrow \mathrm{GL}_3(\mathbb{C})$ such that
$$\zeta_1(\sigma_1) =\left( \begin{array}{@{}c@{}}
  \begin{matrix}
   		d\ & b\ & 0\\\
   		\dfrac{b}{x^2}\ & d\ & 0 \\
   		0\ & 0\ & 1
   		\end{matrix}
\end{array} \right),~ \zeta_1(\rho_1) =\left( \begin{array}{@{}c@{}}
 \begin{matrix}
   		0\ & x \ & 0 \\
   		\dfrac{1}{x} \ & 0\  & 0 \\
   		0 \ & 0\  & 1 
   		\end{matrix}
\end{array} \right),$$
$$\zeta_1(\gamma_1) =\left( \begin{array}{@{}c@{}}
  \begin{matrix}
   		-1\ & 0\ & 0 \\
   		0\  & 1\  & 0 \\
   		0\  & 0\  & 1
   		\end{matrix}
\end{array} \right), \text{ and \ } \zeta_1(\gamma_2) =\left( \begin{array}{@{}c@{}}
 \begin{matrix}
  		1\ & 0\ & 0 \\
   		0\ & -1\ & 0 \\
   		0\  & 0\  & 1 
   		\end{matrix}
\end{array} \right),$$
where $b,d,x \in \mathbb{C}, b^2-d^2x^2\neq 0, x\neq 0$.
\vspace*{0.1cm}
\item[(2)] $\zeta_2: TVB_2 \longrightarrow \mathrm{GL}_3(\mathbb{C})$ such that
$$\zeta_2(\sigma_1) =\left( \begin{array}{@{}c@{}}
  \begin{matrix}
   		\dfrac{2bw+dx}{x}\ & b\ & 0\\
   		\dfrac{b-bw^2}{x^2}\ & d\ & 0\\
   		0\ & 0\ & 1
   		\end{matrix}
\end{array} \right),~ \zeta_2(\rho_1) =\left( \begin{array}{@{}c@{}}
 \begin{matrix}
   		w\ & x\ & 0\\
   		\dfrac{1-w^2}{x}\ & -w\ & 0\\
   		0\ & 0\ & 1
   		\end{matrix}
\end{array} \right),$$
$$\text{and } \zeta_2(\gamma_1)=\zeta_2(\gamma_2)=I_3,$$
where $b,d,w,x \in \mathbb{C}, (d\neq 0 \text{ or } dx\neq \pm b-bw), x\neq 0$.
\vspace*{0.1cm}
\item[(3)] $\zeta_3: TVB_2 \longrightarrow \mathrm{GL}_3(\mathbb{C})$ such that
$$\zeta_3(\sigma_1) =\left( \begin{array}{@{}c@{}}
  \begin{matrix}
   		a\ & b\ & 0\\
   		c\ & d\ & 0\\
   		0\ & 0\ & 1
   		\end{matrix}
\end{array} \right), ~\zeta_3(\rho_1) =\left( \begin{array}{@{}c@{}}
 \begin{matrix}
   		-1\ & 0\ & 0\\
   		0\ & -1\ & 0\\
   		0\ & 0\ & 1
   		\end{matrix}
\end{array} \right), \text{ and \ } \zeta_3(\gamma_1)=\zeta_3(\gamma_2)=I_3,$$
where $a,b,c,d \in \mathbb{C}, ad-bc \neq 0$.
\vspace*{0.1cm}
\item[(4)] $\zeta_4: TVB_2 \longrightarrow \mathrm{GL}_3(\mathbb{C})$ such that
$$\zeta_4(\sigma_1) =\left( \begin{array}{@{}c@{}}
  \begin{matrix}
   		a\ & b\ & 0\\
   		c\ & d\ & 0\\
   		0\ & 0\ & 1
   		\end{matrix}
\end{array} \right),~ \zeta_4(\rho_1) =\left( \begin{array}{@{}c@{}}
 \begin{matrix}
   		1\ & 0\ & 0\\
   		0\ & 1\ & 0\\
   		0\ & 0\ & 1
   		\end{matrix}
\end{array} \right), \text{ and \ } \zeta_4(\gamma_1)=\zeta_4(\gamma_2)=I_3,$$
where $a,b,c,d \in \mathbb{C}, ad-bc \neq 0$.
\vspace*{0.1cm}
\item[(5)] $\zeta_5: TVB_2 \longrightarrow \mathrm{GL}_3(\mathbb{C})$ such that
$$\zeta_5(\sigma_1) =\left( \begin{array}{@{}c@{}}
  \begin{matrix}
   		a\ & 0\ & 0\\
   		c\ & d\ & 0\\
   		0\ & 0\ & 1
   		\end{matrix}
\end{array} \right), ~\zeta_5(\rho_1) =\left( \begin{array}{@{}c@{}}
 \begin{matrix}
   		-1\ & 0\ & 0\\
   		\dfrac{2c}{d-a}\ & 1\ & 0\\
   		0\ & 0\ & 1
   		\end{matrix}
\end{array} \right), \text{ and \ } \zeta_5(\gamma_1)=\zeta_5(\gamma_2)=I_3,$$
where $a,c,d \in \mathbb{C}, ad\neq 0, a\neq d$.
\vspace*{0.1cm}
\item[(6)] $\zeta_6: TVB_2 \longrightarrow \mathrm{GL}_3(\mathbb{C})$ such that
$$\zeta_6(\sigma_1) =\left( \begin{array}{@{}c@{}}
  \begin{matrix}
   		a\ & 0\ & 0\\
   		c\ & d\ & 0\\
   		0\ & 0\ & 1
   		\end{matrix}
\end{array} \right),~ \zeta_6(\rho_1) =\left( \begin{array}{@{}c@{}}
 \begin{matrix}
   		1\ & 0\ & 0\\
   		\dfrac{2c}{a-d}\ & -1\ & 0\\
   		0\ & 0\ & 1
   		\end{matrix}
\end{array} \right), \text{ and \ } \zeta_6(\gamma_1)=\zeta_6(\gamma_2)=I_3,$$
where $a,c,d \in \mathbb{C}, ad\neq 0, a\neq d$.
\vspace*{0.1cm}
\item[(7)] $\zeta_{7}: TVB_2 \longrightarrow \mathrm{GL}_3(\mathbb{C})$ such that
$$\zeta_{7}(\sigma_1) =\left( \begin{array}{@{}c@{}}
  \begin{matrix}
   		d\ & 0\ & 0\\
   		0\ & d\ & 0\\
   		0\ & 0\ & 1
   		\end{matrix}
\end{array} \right),~ \zeta_{7}(\rho_1) =\left( \begin{array}{@{}c@{}}
 \begin{matrix}
   		1\ & 0\ & 0\\
   		y\ & -1\ & 0\\
   		0\ & 0\ & 1
   		\end{matrix}
\end{array} \right), \text{ and \ } \zeta_7(\gamma_1)=\zeta_7(\gamma_2)=I_3,$$
where $d,y \in \mathbb{C}, d\neq 0$.
\item[(8)] $\zeta_{8}: TVB_2 \longrightarrow \mathrm{GL}_3(\mathbb{C})$ such that
$$\zeta_{8}(\sigma_1) =\left( \begin{array}{@{}c@{}}
  \begin{matrix}
   		d\ & 0\ & 0\\
   		0\ & d\ & 0\\
   		0\ & 0\ & 1
   		\end{matrix}
\end{array} \right),~ \zeta_{8}(\rho_1) =\left( \begin{array}{@{}c@{}}
 \begin{matrix}
   		-1\ & 0\ & 0\\
   		y\ & 1\ & 0\\
   		0\ & 0\ & 1
   		\end{matrix}
\end{array} \right), \text{ and \ } \zeta_8(\gamma_1)=\zeta_8(\gamma_2)=I_3,$$
where $d,y \in \mathbb{C}, d\neq 0$.
\end{itemize}
\end{theorem}

\begin{proof}
Set 
$$\zeta(\sigma_1) =\left( \begin{array}{@{}c@{}}
  \begin{matrix}
   		a\ & b\ & 0\\
   		c\ & d\ & 0\\
   		0\ & 0\ & 1
   		\end{matrix}
\end{array} \right), ~\zeta(\rho_1) =\left( \begin{array}{@{}c@{}}
 \begin{matrix}
   		w\ & x\ & 0\\
   		y\ & z\ & 0\\
   		0\ & 0\ & 1
   		\end{matrix}
\end{array} \right),$$
$$\zeta(\gamma_1) =\left( \begin{array}{@{}c@{}}
  \begin{matrix}
   		p\ & q\ & 0\\
   		r\ & s\ & 0\\
   		0\ & 0\ & 1
   		\end{matrix}
\end{array} \right), \text{ and \ } \zeta(\gamma_2) =\left( \begin{array}{@{}c@{}}
 \begin{matrix}
  		1 & 0 & 0\\
   		0 & p & q\\
   		0 & r & s
   		\end{matrix}
\end{array} \right),$$
where $a,b,c,d,w,x,y,z,p,q,r,s \in \mathbb{C}$, $ad-bc\neq 0, wz-xy\neq 0, ps-qr\neq 0$. The only relations of the generators of $TVB_2$ are: 
$$\rho_1^2=1,$$ $$\gamma_1^2=\gamma_2^2=1,$$ $$\gamma_1\gamma_2=\gamma_2\gamma_1,$$ $$\rho_1\gamma_1=\gamma_2\rho_1,$$ $$\rho_1\sigma_1\rho_1=\gamma_2\gamma_1\sigma_1\gamma_1\gamma_2.$$ We apply these relations on the images of the generators of $TVB_2$ under the representation $\zeta$ and we get a system of thirty one equations and twelve unknowns. Solving this system gives directly that $q=r=0$ and $s=1$. Substituting these values in the equations we get,  
\begin{equation}\label{eq1}
w^2+xy =1,
\end{equation}
\begin{equation}\label{eq2}
x(w+z)=0,
\end{equation}
\begin{equation}\label{eq3}
y(w+z)=0,
\end{equation}
\begin{equation}\label{eq4}
z^2+xy=1,
\end{equation}
\begin{equation}\label{eq5}
p^2=1,
\end{equation}
\begin{equation}\label{eq6}
w(1-p)=0,
\end{equation}             
\begin{equation}\label{eq7}
z(1-p)=0,
\end{equation}                   
\begin{equation}\label{eq8}
-ap^2+w(aw+cx)+(bw+dx)y=0,
\end{equation} 
\begin{equation}\label{eq9}
-bp^2+x(aw+cx)+(bw+dx)z=0,
\end{equation}    
\begin{equation}\label{eq10}
-cp^2+w(ay+cz)+y(by+dz)=0,
\end{equation}
\begin{equation}\label{eq11}
-dp^2+x(ay+cz)+z(by+dz)=0.
\end{equation}
From Equation (\ref{eq5}), we get that $p=\pm 1$.

Consider first the case $p=-1$. We get from Equations (\ref{eq6}) and (\ref{eq7}) that $w=z=0$, and so we get from Equation (\ref{eq1}) that $y=\dfrac{1}{x}$. Now, from Equation (\ref{eq8}), we get that $a=d$ and, from Equation (\ref{eq9}), we get that $c=\dfrac{b}{x^2}$. In this case, we obtain that $\zeta$ is equivalent to $\zeta_1$.\vspace{0.1cm}

Now consider the case when $p=1$. This means that the images of $\gamma_1$ and $\gamma_2$ under $\zeta$ are the identity matrix. So, in the remaining cases below, we have $\zeta(\gamma_1)=\zeta(\gamma_2)=I_3$. Now, we subtract Equation (\ref{eq1}) from Equation (\ref{eq4}) to get that $w^2=z^2$, which implies that $w=\pm z$. We consider  each case separately.\vspace{0.1cm}

When $w=z$: From Equations (\ref{eq2}) and (\ref{eq3}), we get that $x=y=0$, and so, from Equations (\ref{eq1}) and (\ref{eq4}), we get that $w=z=\pm 1$. For $w=z=1$, we get that $\zeta$ is equivalent to $\zeta_4$ and for $w=z=-1$, we get that $\zeta$ is equivalent to $\zeta_3$.\vspace{0.1cm}

When $w=-z$: In this case, we consider two sub cases in the following.\vspace{0.1cm}

Case 1: When $x\neq 0$: From Equation (\ref{eq1}), we get that $y=\dfrac{1-w^2}{x}$. Then, direct computations in solving the system of equations (\ref{eq8}), \ldots, (\ref{eq11}) obtain that $a=\dfrac{2bw+dx}{x}$ and $c=\dfrac{b-bw^2}{x^2}$. So, in this case we obtain that $\zeta$ is equivalent to $\zeta_2$.\vspace{0.1cm}

Case 2: When $x=0$: From Equations (\ref{eq1}) and (\ref{eq4}), we get that $w^2=z^2=1$. So, we have $w=\pm1$ and $z=\mp 1$, and so, from Equation (\ref{eq9}), we get that $b=0$. Now, we have two subcases also.\vspace{0.1cm}

Case 2 (a): When $a\neq d$: From Equation (\ref{eq10}), we get that $y=\dfrac{2c}{wa+dz}=\pm \dfrac{2c}{a-d}$. If $y=\dfrac{2c}{a-d}$ then $\zeta$ is equivalent to $\zeta_6$ and if $y=-\dfrac{2c}{a-d}$ then $\zeta$ is equivalent to $\zeta_5$.\vspace{0.1cm}

Case 2 (b): When $a=d$: From Equation (\ref{eq10}), we get that $c=0$. So, if $w=1$ and $z=-1$ then $\zeta$ is equivalent to $\zeta_7$ and if $w=-1$ and $z=1$ then $\zeta$ is equivalent to $\zeta_8$.\vspace{0.1cm}

Therefore, $\zeta$ is equivalent to one of the representations $\zeta_i, 1\leq i \leq 8,$ as required.\\
\end{proof}

Now, we study the faithfulness of the representations $\zeta_i, 1\leq i \leq 8,$ given in Theorem \ref{localTVB2}.

\begin{theorem}
Every representation of type $\zeta_i, 1\leq i \leq 8,$ is unfaithful.
\end{theorem}
\begin{proof}
We consider the following two cases:

Case 1: The representations of type $\zeta_1$: In this case we have two unequal elements  $\sigma_1\rho_1$ and $ \rho_1\sigma_1$ of $ TVB_2$ but their images under $\zeta_1$ map to the same matrix. So $\zeta_1$ is unfaithful. 

Case 2: The representations of type $\zeta_i, 2\leq i \leq 8$. In this case $\gamma_1$ and $\gamma_2$ map to the identity matrix. So $\zeta_i, 2\leq i \leq 8,$ is unfaithful.
\end{proof}

We now study the irreducibility of the representations $\zeta_i, 1\leq i \leq 8,$ given in Theorem \ref{localTVB2}.

\begin{theorem}
All the representations in Theorem 3.1 are reducible to a two-dimensional representation and the following statements hold true.
\begin{enumerate}
    \item[(i)] The representation $\zeta_1$ is not further reducible to a one-dimensional representation.
    \item[(ii)] The representation $\zeta_2$ is further reducible to a one-dimensional representation. Moreover, $\zeta_2$ is equivalent to a direct sum of three one-dimensional representations.
    \item[(iii)] The representations $\zeta_i$, $3\leq i \leq 8,$ are further reducible to a one-dimensional representation.
\end{enumerate}
\end{theorem}

\begin{proof}
We first observe that all the given representations are reducible to a two-dimensional representation. This is because, in each representation matrix, the third standard basis vector \( e_3 \) is preserved—meaning that it is a common eigenvector for all the matrices. As a result, the three-dimensional representation space admits a one-dimensional invariant subspace spanned by \( e_3 \), and the representation decomposes into a direct sum of this one-dimensional subrepresentation and a complementary two-dimensional subspace spanned by $e_1$ and $e_2$. Therefore, each representation is reducible and can be expressed in block form with a \(2 \times 2\) subrepresentation.\vspace{0.1cm}

To determine whether this \(2 \times 2\) subrepresentation is further reducible, we examine whether all the corresponding \(2 \times 2\) matrices share a common eigenvector. Note that if the $2\times 2$ composition factor of a matrix is $\pm I_2$, then every nonzero vector is an eigenvector of it, so such a matrix imposes no condition.\vspace{0.1cm}

\text{Case 1:} The representations of type $\zeta_1$. In this case, the $2 \times 2$ composition factor of $\zeta_1(\gamma_1)$ is
$\begin{pmatrix}
-1 & 0 \\
0 & 1
\end{pmatrix}$.
Since its eigenvalues $-1$ and $1$ are distinct, its only eigenvectors are the nonzero multiples of
$\begin{pmatrix}
1 \\
0
\end{pmatrix}$
and
$\begin{pmatrix}
0 \\
1
\end{pmatrix}$.
On the other hand, the $2\times 2$ composition factor of $\zeta_1(\rho_1)$ satisfies
\[
\begin{pmatrix}
0 & x \\
\dfrac{1}{x} & 0
\end{pmatrix}
\begin{pmatrix}
1 \\
0
\end{pmatrix}
=
\begin{pmatrix}
0 \\
\dfrac{1}{x}
\end{pmatrix}
\quad \text{and} \quad
\begin{pmatrix}
0 & x \\
\dfrac{1}{x} & 0
\end{pmatrix}
\begin{pmatrix}
0 \\
1
\end{pmatrix}
=
\begin{pmatrix}
x \\
0
\end{pmatrix}.
\]
Since $x\neq 0$, neither $\begin{pmatrix}1 \\ 0\end{pmatrix}$ nor $\begin{pmatrix}0 \\ 1\end{pmatrix}$ is an eigenvector of $\zeta_1(\rho_1)$. So, the matrices $\zeta_1(\gamma_1)$ and $\zeta_1(\rho_1)$ have no common eigenvector, for all values of $b$ and $d$. Hence, $\zeta_1$ is not further reducible to a one-dimensional representation.\vspace{0.1cm}

Case 2: The representations of type $\zeta_2$. The $2 \times 2$ composition factors of $\zeta_2(\gamma_1)$ and $\zeta_2(\gamma_2)$ are $I_2$, so it suffices to consider $\zeta_2(\sigma_1)$ and $\zeta_2(\rho_1)$. The $2 \times 2$ composition factor of $\zeta_2(\rho_1)$ has trace $0$ and determinant $-w^2-(1-w^2)=-1$, so its eigenvalues are $1$ and $-1$, with eigenvectors
$
u_1=\begin{pmatrix}
x \\
1-w
\end{pmatrix}
$
and
$
u_2=\begin{pmatrix}
x \\
-(1+w)
\end{pmatrix},
$
respectively. Direct computations show that
\begin{gather*}
\begin{pmatrix}
\dfrac{2bw+dx}{x} & b \\
\dfrac{b-bw^2}{x^2} & d
\end{pmatrix}
u_1=\left(d+\dfrac{b(1+w)}{x}\right)u_1
\quad \text{and} \\
\begin{pmatrix}
\dfrac{2bw+dx}{x} & b \\
\dfrac{b-bw^2}{x^2} & d
\end{pmatrix}
u_2=\left(d-\dfrac{b(1-w)}{x}\right)u_2.
\end{gather*}
Thus, $u_1$ and $u_2$ are common eigenvectors of all the $2\times 2$ composition factors, and so $\zeta_2$ is further reducible to a one-dimensional representation. Moreover, $\det(u_1\ u_2)=-2x\neq 0$, so $u_1$ and $u_2$ are linearly independent. Setting
$
P=\begin{pmatrix}
x & x & 0\\
1-w & -(1+w) & 0\\
0 & 0 & 1
\end{pmatrix},
$
we get that
\begin{gather*}
P^{-1}\zeta_2(\sigma_1)P=\begin{pmatrix}
d+\dfrac{b(1+w)}{x} & 0 & 0\\
0 & d-\dfrac{b(1-w)}{x} & 0\\
0 & 0 & 1
\end{pmatrix},\\
P^{-1}\zeta_2(\rho_1)P=\begin{pmatrix}
1 & 0 & 0\\
0 & -1 & 0\\
0 & 0 & 1
\end{pmatrix},
\end{gather*}
and $P^{-1}\zeta_2(\gamma_1)P=P^{-1}\zeta_2(\gamma_2)P=I_3$. Hence, $\zeta_2$ is equivalent to a direct sum of three one-dimensional representations.\vspace{0.1cm}

\text{Case 3:} The representations of type $\zeta_3$. The $2 \times 2$ composition factor of $\zeta_3(\rho_1)$ is $-I_2$, and those of $\zeta_3(\gamma_1)$ and $\zeta_3(\gamma_2)$ are $I_2$. So, every nonzero vector is an eigenvector of these matrices, and it suffices to consider $\zeta_3(\sigma_1)$. Since we are working over $\mathbb{C}$, the $2 \times 2$ composition factor
$\begin{pmatrix}
a & b \\
c & d
\end{pmatrix}$
of $\zeta_3(\sigma_1)$ has an eigenvalue $\lambda$, which is a root of $t^2-(a+d)t+(ad-bc)$, and hence an eigenvector $v$. Explicitly, one can take
$v=\begin{pmatrix}
b \\
\lambda-a
\end{pmatrix}$
if $b\neq 0$, and
$v=\begin{pmatrix}
0 \\
1
\end{pmatrix}$
if $b=0$. Thus, $v$ is a common eigenvector of all the $2 \times 2$ composition factors, and so $\zeta_3$ is further reducible to a one-dimensional representation for all values of $a,b,c,d$. We note that $\zeta_3$ is equivalent to a direct sum of three one-dimensional representations if and only if the matrix
$\begin{pmatrix}
a & b \\
c & d
\end{pmatrix}$
is diagonalizable.\vspace{0.1cm}

\text{Case 4:} The representations of type $\zeta_4$. This case is \text{similar to Case 3} and follows the same reasoning, since the $2\times 2$ composition factor of $\zeta_4(\rho_1)$ is $I_2$.\vspace{0.1cm}

\text{Case 5:} The representations of type $\zeta_5$. All $2 \times 2$ composition factors of the matrices $\zeta_5(\sigma_1)$, $\zeta_5(\rho_1)$, $\zeta_5(\gamma_1)$, and $\zeta_5(\gamma_2)$ share a common eigenvector
$
\begin{pmatrix}
0 \\
1
\end{pmatrix}.
$
Thus, $\zeta_5$ is \text{reducible to a one-dimensional representation}.\vspace{0.1cm}

Case 6: The representations of type $\zeta_i, i=6,7,8$. This is the same as Case 5; $\zeta$ is reducible to a one-dimensional representation.

\end{proof}

In the following, we determine the forms of all complex homogeneous local representations of $TVB_n$ for all $n\geq 3$.

\begin{theorem}\label{localTVBn}
Consider $n\geq 3$ and let $\zeta': TVB_n \longrightarrow \mathrm{GL}_{n+1}(\mathbb{C})$ be a complex homogeneous local representation of $TVB_n$. Then, $\zeta'$ is equivalent to one of the following seven representations.
\begin{itemize}
\item[(1)] $\zeta'_1: TVB_n \longrightarrow \mathrm{GL}_{n+1}(\mathbb{C})$ such that
$$\zeta'_1(\sigma_i)=\left( \begin{array}{c|@{}c|c@{}}
   \begin{matrix}
     I_{i-1} 
   \end{matrix} 
      & \textbf{0} & \textbf{0} \\
      \hline
    \textbf{0} &\hspace{0.2cm} \begin{matrix}
   		0\ & b\ \\
   		c\ & 0\ \\
   		\end{matrix}  & \textbf{0}  \\
\hline
\textbf{0} & \textbf{0} & I_{n-i}
\end{array} \right), \zeta'_1(\rho_i)=\left( \begin{array}{c|@{}c|c@{}}
   \begin{matrix}
     I_{i-1} 
   \end{matrix} 
      & \textbf{0} & \textbf{0} \\
      \hline
    \textbf{0} &\hspace{0.2cm} \begin{matrix}
   		0\ & -\dfrac{\sqrt{b}}{\sqrt{c}}\ \\
   		-\dfrac{\sqrt{c}}{\sqrt{b}}\ & 0\ \\
   		\end{matrix}  & \textbf{0}  \\
\hline
\textbf{0} & \textbf{0} & I_{n-i}
\end{array} \right)$$
$$\text{ and \ } \zeta'_1(\gamma_j)=\left( \begin{array}{c|@{}c|c@{}}
   \begin{matrix}
     I_{j-1} 
   \end{matrix} 
      & \textbf{0} & \textbf{0} \\
      \hline
    \textbf{0} &\hspace{0.2cm} \begin{matrix}
   		-1\ & 0\ \\
   		0\ & 1\ \\
   		\end{matrix}  & \textbf{0}  \\
\hline
\textbf{0} & \textbf{0} & I_{n-j}
\end{array} \right)$$
for $1\leq i \leq n-1,\ 1 \leq j \leq n,$ where $b,c\in \mathbb{C}^*.$
\vspace*{0.1cm}
\item[(2)] $\zeta'_2: TVB_n \longrightarrow \mathrm{GL}_{n+1}(\mathbb{C})$ such that
$$\zeta'_2(\sigma_i)=\left( \begin{array}{c|@{}c|c@{}}
   \begin{matrix}
     I_{i-1} 
   \end{matrix} 
      & \textbf{0} & \textbf{0} \\
      \hline
    \textbf{0} &\hspace{0.2cm} \begin{matrix}
   		0\  & b\ \\
   		c\  & 0\ \\
   		\end{matrix}  & \textbf{0}  \\
\hline
\textbf{0} & \textbf{0} & I_{n-i}
\end{array} \right), \zeta'_2(\rho_i)=\left( \begin{array}{c|@{}c|c@{}}
   \begin{matrix}
     I_{i-1} 
   \end{matrix} 
      & \textbf{0} & \textbf{0} \\
      \hline
    \textbf{0} &\hspace{0.2cm} \begin{matrix}
   		0\ & \dfrac{\sqrt{b}}{\sqrt{c}}\ \\
   		\dfrac{\sqrt{c}}{\sqrt{b}}\ & 0\ \\
   		\end{matrix}  & \textbf{0}  \\
\hline
\textbf{0} & \textbf{0} & I_{n-i}
\end{array} \right)$$
$$
\text{ and \ } \zeta'_2(\gamma_j)=\left( \begin{array}{c|@{}c|c@{}}
   \begin{matrix}
     I_{j-1} 
   \end{matrix} 
      & \textbf{0} & \textbf{0} \\
      \hline
    \textbf{0} &\hspace{0.2cm} \begin{matrix}
   		-1\  & 0\ \\
   		0\  & 1\ \\
   		\end{matrix}  & \textbf{0}  \\
\hline
\textbf{0} & \textbf{0} & I_{n-j}
\end{array} \right)$$
for $1\leq i \leq n-1,\ 1 \leq j \leq n,$ where $b,c\in \mathbb{C}^*.$
\vspace*{0.1cm}
\item[(3)] $\zeta'_3: TVB_n \longrightarrow \mathrm{GL}_{n+1}(\mathbb{C})$ such that
$$\zeta'_3(\sigma_i)=\left( \begin{array}{c|@{}c|c@{}}
   \begin{matrix}
     I_{i-1} 
   \end{matrix} 
      & \textbf{0} & \textbf{0} \\
      \hline
    \textbf{0} &\hspace{0.2cm} \begin{matrix}
   		0\ & b\ \\
   		c\  & 0\ \\
   		\end{matrix}  & \textbf{0}  \\
\hline
\textbf{0} & \textbf{0} & I_{n-i}
\end{array} \right), \zeta'_3(\rho_i)=\left( \begin{array}{c|@{}c|c@{}}
   \begin{matrix}
     I_{i-1} 
   \end{matrix} 
      & \textbf{0} & \textbf{0} \\
      \hline
    \textbf{0} &\hspace{0.2cm} \begin{matrix}
   		0\ & -\dfrac{\sqrt{b}}{\sqrt{c}}\ \\
   		-\dfrac{\sqrt{c}}{\sqrt{b}}\  & 0\ \\
   		\end{matrix}  & \textbf{0}  \\
\hline
\textbf{0} & \textbf{0} & I_{n-i}
\end{array} \right)$$
$$\text{and } \zeta'_3(\gamma_j)=I_{n+1}$$
for $1\leq i \leq n-1,\ 1 \leq j \leq n,$ where $b,c\in \mathbb{C}^*.$
\vspace*{0.1cm}
\item[(4)] $\zeta'_4: TVB_n \longrightarrow \mathrm{GL}_{n+1}(\mathbb{C})$ such that
$$\zeta'_4(\sigma_i)=\left( \begin{array}{c|@{}c|c@{}}
   \begin{matrix}
     I_{i-1} 
   \end{matrix} 
      & \textbf{0} & \textbf{0} \\
      \hline
    \textbf{0} &\hspace{0.2cm} \begin{matrix}
   		0\ & b\ \\
   		c\  & 0\ \\
   		\end{matrix}  & \textbf{0}  \\
\hline
\textbf{0} & \textbf{0} & I_{n-i}
\end{array} \right), \zeta'_4(\rho_i)=\left( \begin{array}{c|@{}c|c@{}}
   \begin{matrix}
     I_{i-1} 
   \end{matrix} 
      & \textbf{0} & \textbf{0} \\
      \hline
    \textbf{0} &\hspace{0.2cm} \begin{matrix}
   		0\ & \dfrac{\sqrt{b}}{\sqrt{c}}\ \\
   		\dfrac{\sqrt{c}}{\sqrt{b}}\ & 0\ \\
   		\end{matrix}  & \textbf{0}  \\
\hline
\textbf{0} & \textbf{0} & I_{n-i}
\end{array} \right)$$
$$\text{and } \zeta'_4(\gamma_j)=I_{n+1}$$ 
for $1\leq i \leq n-1, \ 1 \leq j \leq n,$ where $b,c\in \mathbb{C}^*.$
\vspace*{0.1cm}
\item[(5)] $\zeta'_5: TVB_n \longrightarrow \mathrm{GL}_{n+1}(\mathbb{C})$ such that
$$\zeta'_5(\sigma_i)=I_{n+1},\ \zeta'_5(\rho_i)=\left( \begin{array}{c|@{}c|c@{}}
   \begin{matrix}
     I_{i-1} 
   \end{matrix} 
      & \textbf{0} & \textbf{0} \\
      \hline
    \textbf{0} &\hspace{0.2cm} \begin{matrix}
   		0\ & x\ \\
   		\dfrac{1}{x}\  & 0\ \\
   		\end{matrix}  & \textbf{0}  \\
\hline
\textbf{0} & \textbf{0} & I_{n-i}
\end{array} \right)$$
$$\text{and } \zeta'_5(\gamma_j)=\left( \begin{array}{c|@{}c|c@{}}
   \begin{matrix}
     I_{j-1} 
   \end{matrix} 
      & \textbf{0} & \textbf{0} \\
      \hline
    \textbf{0} &\hspace{0.2cm} \begin{matrix}
   		-1\ & 0\ \\
   		0 \ & 1\ \\
   		\end{matrix}  & \textbf{0}  \\
\hline
\textbf{0} & \textbf{0} & I_{n-j}
\end{array} \right)$$
$\text{ for } 1\leq i \leq n-1, \ 1 \leq j \leq n, \text{ where } x\in \mathbb{C}^*.$
\vspace*{0.1cm}
\item[(6)] $\zeta'_6: TVB_n \longrightarrow \mathrm{GL}_{n+1}(\mathbb{C})$ such that
$$\zeta'_6(\sigma_i)=I_{n+1},\  \zeta'_6(\rho_i)=\left( \begin{array}{c|@{}c|c@{}}
   \begin{matrix}
     I_{i-1} 
   \end{matrix} 
      & \textbf{0} & \textbf{0} \\
      \hline
    \textbf{0} &\hspace{0.2cm} \begin{matrix}
   		0\ & x\ \\
   		\dfrac{1}{x}\  & 0\ \\
   		\end{matrix}  & \textbf{0}  \\
\hline
\textbf{0} & \textbf{0} & I_{n-i}
\end{array} \right)
\text{ and \ } \zeta'_6(\gamma_j)=I_{n+1},$$
$\text{ for } 1\leq i \leq n-1, 1 \leq j \leq n, \text{ where } x \in \mathbb{C}^*.$
\vspace*{0.1cm}
\item[(7)] $\zeta'_7: TVB_n \longrightarrow \mathrm{GL}_{n+1}(\mathbb{C})$ such that
$$\zeta'_7(\sigma_i)=\zeta'_7(\rho_i)=\zeta'_7(\gamma_j)=I_{n+1} \text{ for } 1\leq i \leq n-1, \ 1 \leq j \leq n.$$\\
\end{itemize}
\end{theorem} 

\begin{proof}
Since $\zeta'$ is a homogeneous local representation of $TVB_n$, we may set
$$\zeta'(\sigma_i)=\left( \begin{array}{c|@{}c|c@{}}
   \begin{matrix}
     I_{i-1} 
   \end{matrix} 
      & \textbf{0} & \textbf{0} \\
      \hline
    \textbf{0} &\hspace{0.2cm} \begin{matrix}
   		a\ & b\ \\
   		c \ & d\ \\
   		\end{matrix}  & \textbf{0}  \\
\hline
\textbf{0} & \textbf{0} & I_{n-i}
\end{array} \right),$$ 
$$\zeta'(\rho_i)=\left( \begin{array}{c|@{}c|c@{}}
   \begin{matrix}
     I_{i-1} 
   \end{matrix} 
      & \textbf{0} & \textbf{0} \\
      \hline
    \textbf{0} &\hspace{0.2cm} \begin{matrix}
   		w\ & x\ \\
   		y \ & z\ \\
   		\end{matrix}  & \textbf{0}  \\
\hline
\textbf{0} & \textbf{0} & I_{n-i}
\end{array} \right),$$
and
$$\zeta'(\gamma_j)=\left( \begin{array}{c|@{}c|c@{}}
   \begin{matrix}
     I_{j-1} 
   \end{matrix} 
      & \textbf{0} & \textbf{0} \\
      \hline
    \textbf{0} &\hspace{0.2cm} \begin{matrix}
   		p\  & q\ \\
   		r\  & s\ \\
   		\end{matrix}  & \textbf{0}  \\
\hline
\textbf{0} & \textbf{0} & I_{n-j}
\end{array} \right)$$
$\text{for } 1\leq i \leq n-1, 1 \leq j \leq n, \text{ where } a,b,c,d,w,x,y,z,p,q,r,s\in \mathbb{C}$ and the matrices supposed to be invertible. Now, according to the locality and homogeneity of the representation, it suffices to consider the following relations of the generators of $TVB_n$, and all other relations imply similar equations.
\begin{align*}
\sigma_1\sigma_2\sigma_1 &= \sigma_2\sigma_1\sigma_2, \\
\rho_1^2 &=1, \\
\rho_1\rho_2\rho_1 &=\rho_2\rho_1\rho_2,\\
\rho_1\rho_2\sigma_1 &=\sigma_2\rho_1\rho_2,\\
\gamma_1^2 &=1,\\
\gamma_1\gamma_2 &=\gamma_2\gamma_1,\\
\gamma_3\sigma_1 &=\sigma_1\gamma_3,\\
\gamma_3\rho_1 &=\rho_1\gamma_3,\\
\rho_1\gamma_1 &=\gamma_2\rho_1,\\
\rho_1\sigma_1\rho_1 &=\gamma_2\gamma_1\sigma_1\gamma_1\gamma_2.
\end{align*}
Applying these relations to the images of the generators of $TVB_n$ under the representation $\zeta'$ leads to a system of equations with twelve unknowns. A similar proof as the proof of Theorem \ref{localTVB2} gives the required result.
\end{proof}

Now, we study the faithfulness of the representations $\zeta'_i, 1\leq i \leq 7,$ given in Theorem \ref{localTVBn}.

\begin{theorem}
Every representation of type $\zeta'_i, 1\leq i \leq 7,$ is unfaithful.
\end{theorem}
\begin{proof}
    The representations \( \zeta'_1 \) and \( \zeta'_2 \) are unfaithful, as the elements \( \gamma_1 \sigma_1 \gamma_2 \) and \( \gamma_2 \sigma_1 \gamma_1 \) are distinct in \( TVB_n \), yet their images under both \( \zeta'_1 \) and \( \zeta'_2 \) are the same matrix. Moreover, the representations \( \zeta'_i \), for \( 3 \leq i \leq 7 \), are also unfaithful. Specifically, when \( i = 3, 4, 6, 7 \), the generators \( \gamma_j \) are mapped to the identity matrix, and when \( i = 5 \), the generators \( \sigma_j \) are mapped to the identity matrix.
\end{proof}

We now study the irreducibility of the representations $\zeta'_i, 1\leq i \leq 7,$ given in Theorem \ref{localTVBn}. First, we need the following lemma introduced in \cite{NASSER2025}.

\begin{lemma}
Consider $n \geq 2$ and let \( \lambda : G \longrightarrow \mathrm{GL}_n(V) \) be a representation of a group G on a vector space V. Let H be any subgroup of G. If \(\lambda : H \longrightarrow \mathrm{GL}_n(V) \) is irreducible, then \( \lambda : G \longrightarrow \mathrm{GL}_n(V) \) is also irreducible.
\end{lemma}

\begin{theorem}
Every representation of type $\zeta'_i, 1\leq i \leq 7,$ is reducible to an $n$-dimensional representation. Furthermore, the representations \( \zeta_i' \), for \( 1 \leq i \leq 4 \), are not further reducible to an $(n-1)$-dimensional representation if $bc \neq 1$.
\end{theorem}
\begin{proof}
In each representation, the standard basis vector $e_{n+1}$ is fixed by every matrix, so it is a common eigenvector spanning a one-dimensional invariant subspace. Hence each $\zeta'_i$ is reducible to an $n$-dimensional representation.

To analyze the further reducibility of the representations \( \zeta_i' \), for \( 1 \leq i \leq 4 \), we consider the \( n \times n \) composition factors of the matrices \( \zeta_i'(\sigma_j) \), \( \zeta_i'(\rho_j) \), and \( \zeta_i'(\gamma_j) \). Since \( B_n \) is a subgroup of \( TVB_n \), we can restrict each \( \zeta_i' \) to \( B_n \). By Theorem 3.3 in \cite{chreif2024}, and using the previous lemma, the resulting degree \( n \) representation is irreducible if and only if \( bc \neq 1 \). Therefore, the representations \( \zeta_i' \), $1\leq i \leq 4$, are not further reducible to an $(n-1)$-dimensional representation if \( bc \neq 1 \).
\end{proof}

\section{On Complex Local Representations of $STVB_2$ and $STVG_2$}

In this section, we determine the forms of all complex homogeneous local representations of $STVB_2$ and $STVG_2$. In addition, we study the irreducibility and the faithfulness for these representations.

\begin{theorem} \label{localSTVB2}
Let $\eta: STVB_2 \longrightarrow \mathrm{M}_3(\mathbb{C})$ be a complex homogeneous local representation of $STVB_2$. Then, $\eta$ is equivalent to one of the following thirteen representations. \\
\begin{itemize}
\item[(1)] $\eta_1: STVB_2 \longrightarrow \mathrm{M}_3(\mathbb{C})$ such that
$$\eta_1(\sigma_1) =\left( \begin{array}{@{}c@{}}
  \begin{matrix}
   		a\ & b\  & 0\\
   		\dfrac{b}{x^2}\ & a\ & 0\\
   		0\ & 0\ & 1
   		\end{matrix}
\end{array} \right), \eta_1(\rho_1) =\left( \begin{array}{@{}c@{}}
 \begin{matrix}
   		0\ & x\ & 0\\
   		\dfrac{1}{x}\ & 0\ & 0\\
   		0\ & 0\ & 1
   		\end{matrix}
\end{array} \right), \eta_1(\tau_1) =\left( \begin{array}{@{}c@{}}
 \begin{matrix}
   		f\ & g\ & 0\\
   		\dfrac{g}{x^2}\ & f\ & 0\\
   		0\ & 0\ & 1
   		\end{matrix}
\end{array} \right),$$
$$\eta_1(\gamma_1) =\left( \begin{array}{@{}c@{}}
  \begin{matrix}
   		-1\ & 0\ & 0\\
   		0\ & 1\ & 0\\
   		0\ & 0\ & 1
   		\end{matrix}
\end{array} \right) \text{ and \ } \eta_1(\gamma_2) =\left( \begin{array}{@{}c@{}}
 \begin{matrix}
  		1\ & 0\ & 0\\
   		0\ & -1\ & 0\\
   		0\ & 0\ & 1
   		\end{matrix}
\end{array} \right),$$
where $a,b,x,f,g \in \mathbb{C}, a^2x^2-b^2\neq 0, x\neq 0$.
\vspace*{0.1cm}
\item[(2)] $\eta_2: STVB_2 \longrightarrow \mathrm{M}_3(\mathbb{C})$ such that
$$\eta_2(\sigma_1) =\left( \begin{array}{@{}c@{}}
  \begin{matrix}
   		a\ & b\ & 0\\
   		\dfrac{b-bz^2}{x^2}\ & \dfrac{ax+2bz}{x}\ & 0\\
   		0\ & 0\ & 1
   		\end{matrix}
\end{array} \right), \eta_2(\rho_1) =\left( \begin{array}{@{}c@{}}
 \begin{matrix}
   		-z\ & x\ & 0\\
   		\dfrac{1-z^2}{x}\ & z\ & 0\\
   		0\ & 0\ & 1
   		\end{matrix}
\end{array} \right),$$
$$\eta_2(\tau_1) =\left( \begin{array}{@{}c@{}}
 \begin{matrix}
   		f\ & g\ & 0\\
   		\dfrac{g-gz^2}{x^2}\ & \dfrac{fx+2gz}{x}\ & 0\\
   		0\ & 0\ & 1
   		\end{matrix}
\end{array} \right) \text{ and \ } \eta_2(\gamma_1)=\eta_2(\gamma_2)=I_3,$$
where $a,b,x,z,f,g \in \mathbb{C}, a^2x^2+2abxz-b^2+b^2z^2\neq 0, x\neq 0.$
\vspace*{0.1cm}
\item[(3)] $\eta_3: STVB_2 \longrightarrow \mathrm{M}_3(\mathbb{C})$ such that
$$\eta_3(\sigma_1) =\left( \begin{array}{@{}c@{}}
  \begin{matrix}
   		a\ & b\ & 0\\
   		c\ & d\ & 0\\
   		0\ & 0\ & 1
   		\end{matrix}
\end{array} \right), \eta_3(\rho_1) =\left( \begin{array}{@{}c@{}}
 \begin{matrix}
   		-1\ & 0\ & 0\\
   		0\ & -1\ & 0\\
   		0\ & 0\ & 1
   		\end{matrix}
\end{array} \right),$$ 
$$\eta_3(\tau_1) =\left( \begin{array}{@{}c@{}}
 \begin{matrix}
   		f\ & g\ & 0\\
   		\dfrac{cg}{b}\ & \dfrac{bf-ag+dg}{b}\ & 0\\
   		0\ & 0\ & 1
   		\end{matrix}
\end{array} \right) \text{ and \ } \eta_3(\gamma_1)=\eta_3(\gamma_2)=I_3,$$
where $a,b,c,d,f,g \in \mathbb{C}, ad-bc\neq 0, b\neq 0.$
\vspace*{0.1cm}
\item[(4)] $\eta_4: STVB_2 \longrightarrow \mathrm{M}_3(\mathbb{C})$ such that
$$\eta_4(\sigma_1) =\left( \begin{array}{@{}c@{}}
  \begin{matrix}
   		a\ & b\ & 0\\
   		c\ & d\ & 0\\
   		0\ & 0\ & 1
   		\end{matrix}
\end{array} \right), \eta_4(\rho_1) =\left( \begin{array}{@{}c@{}}
 \begin{matrix}
   		1\ & 0\ & 0\\
   		0\ & 1\ & 0\\
   		0\ & 0\ & 1
   		\end{matrix}
\end{array} \right),$$ 
$$\eta_4(\tau_1) =\left( \begin{array}{@{}c@{}}
 \begin{matrix}
   		f\ & g\ & 0\\
   		\dfrac{cg}{b}\ & \dfrac{bf-ag+dg}{b}\ & 0\\
   		0\ & 0\ & 1
   		\end{matrix}
\end{array} \right) \text{ and \ } \eta_4(\gamma_1)=\eta_4(\gamma_2)=I_3,$$
where $a,b,c,d,f,g \in \mathbb{C}, ad-bc\neq 0, b\neq 0$.
\vspace*{0.1cm}
\item[(5)] $\eta_5: STVB_2 \longrightarrow \mathrm{M}_3(\mathbb{C})$ such that
$$\eta_5(\sigma_1) =\left( \begin{array}{@{}c@{}}
  \begin{matrix}
   		a\ & 0\ & 0\\
   		0\ & a\ & 0\\
   		0\ & 0\ & 1
   		\end{matrix}
\end{array} \right), \eta_5(\rho_1) =\left( \begin{array}{@{}c@{}}
 \begin{matrix}
   		-1\ & 0\ & 0\\
   		0\ & -1\ & 0\\
   		0\ & 0\ & 1
   		\end{matrix}
\end{array} \right),$$ 
$$\eta_5(\tau_1) =\left( \begin{array}{@{}c@{}}
 \begin{matrix}
   		f\ & g\ & 0\\
   		h\ & k\ & 0\\
   		0\ & 0\ & 1
   		\end{matrix}
\end{array} \right) \text{ and \ } \eta_5(\gamma_1)=\eta_5(\gamma_2)=I_3,$$
where $a,f,g,h,k \in \mathbb{C}, a\neq 0$.
\vspace*{0.1cm}
\item[(6)] $\eta_6: STVB_2 \longrightarrow \mathrm{M}_3(\mathbb{C})$ such that
$$\eta_6(\sigma_1) =\left( \begin{array}{@{}c@{}}
  \begin{matrix}
   		a\ & 0\ & 0\\
   		c\ & d\ & 0\\
   		0\ & 0\ & 1
   		\end{matrix}
\end{array} \right), \eta_6(\rho_1) =\left( \begin{array}{@{}c@{}}
 \begin{matrix}
   		-1\ & 0\ & 0\\
   		0\ & -1\ & 0\\
   		0\ & 0\ & 1
   		\end{matrix}
\end{array} \right),$$ 
$$\eta_6(\tau_1) =\left( \begin{array}{@{}c@{}}
 \begin{matrix}
   		f\ & 0\ & 0\\
   		h\ & \dfrac{cf-ah+dh}{c}\ & 0\\
   		0\ & 0\ & 1
   		\end{matrix}
\end{array} \right) \text{ and \ } \eta_6(\gamma_1)=\eta_6(\gamma_2)=I_3,$$
where $a,c,d,f,h \in \mathbb{C}, ad\neq 0, c\neq 0$.
\vspace*{0.1cm}
\item[(7)] $\eta_7: STVB_2 \longrightarrow \mathrm{M}_3(\mathbb{C})$ such that
$$\eta_7(\sigma_1) =\left( \begin{array}{@{}c@{}}
  \begin{matrix}
   		a\ & 0\ & 0\\
   		c\ & \dfrac{-2c+ay}{y}\ & 0\\
   		0\ & 0\ & 1
   		\end{matrix}
\end{array} \right), \eta_7(\rho_1) =\left( \begin{array}{@{}c@{}}
 \begin{matrix}
   		1\ & 0\ & 0\\
   		y\ & -1\ & 0\\
   		0\ & 0\ & 1
   		\end{matrix}
\end{array} \right),$$ 
$$\eta_7(\tau_1) =\left( \begin{array}{@{}c@{}}
 \begin{matrix}
   		f\ & 0\ & 0\\
   		h\ & \dfrac{-2h+fy}{y}\ & 0\\
   		0\ & 0\ & 1
   		\end{matrix}
\end{array} \right) \text{ and \ } \eta_7(\gamma_1)=\eta_7(\gamma_2)=I_3,$$
where $a,c,y,f,h \in \mathbb{C},a(-2c+ay)\neq 0 ,y\neq 0$.
\vspace*{0.1cm}
\item[(8)] $\eta_8: STVB_2 \longrightarrow \mathrm{M}_3(\mathbb{C})$ such that
$$\eta_8(\sigma_1) =\left( \begin{array}{@{}c@{}}
  \begin{matrix}
   		a\ & 0\ & 0\\
   		c\ & \dfrac{2c+ay}{y}\ & 0\\
   		0\ & 0\ & 1
   		\end{matrix}
\end{array} \right), \eta_8(\rho_1) =\left( \begin{array}{@{}c@{}}
 \begin{matrix}
   		-1\ & 0\ & 0\\
   		y\ & 1\ & 0\\
   		0\ & 0\ & 1
   		\end{matrix}
\end{array} \right),$$ 
$$\eta_8(\tau_1) =\left( \begin{array}{@{}c@{}}
 \begin{matrix}
   		f\ & 0\ & 0\\
   		h\ & \dfrac{2h+fy}{y}\ & 0\\
   		0\ & 0\ & 1
   		\end{matrix}
\end{array} \right) \text{ and \ } \eta_8(\gamma_1)=\eta_8(\gamma_2)=I_3,$$
where $a,c,y,f,h \in \mathbb{C},a(2c+ay)\neq 0 ,y\neq 0$.
\vspace*{0.1cm}
\item[(9)] $\eta_9: STVB_2 \longrightarrow \mathrm{M}_3(\mathbb{C})$ such that
$$\eta_9(\sigma_1) =\left( \begin{array}{@{}c@{}}
  \begin{matrix}
   		a\ & 0\ & 0\\
   		0\ & a\ & 0\\
   		0\ & 0\ & 1
   		\end{matrix}
\end{array} \right), \eta_9(\rho_1) =\left( \begin{array}{@{}c@{}}
 \begin{matrix}
   		1\ & 0\ & 0\\
   		0\ & 1\ & 0\\
   		0\ & 0\ & 1
   		\end{matrix}
\end{array} \right),$$ 
$$\eta_9(\tau_1) =\left( \begin{array}{@{}c@{}}
 \begin{matrix}
   		f\ & g\ & 0\\
   		h\ & k\ & 0\\
   		0\ & 0\ & 1
   		\end{matrix}
\end{array} \right) \text{ and \ } \eta_9(\gamma_1)=\eta_9(\gamma_2)=I_3,$$
where $a,f,g,h,k \in \mathbb{C}, a\neq 0$.
\vspace*{0.1cm}
\item[(10)] $\eta_{10}: STVB_2 \longrightarrow \mathrm{M}_3(\mathbb{C})$ such that
$$\eta_{10}(\sigma_1) =\left( \begin{array}{@{}c@{}}
  \begin{matrix}
   		a\ & 0\ & 0\\
   		c\ & d\ & 0\\
   		0\ & 0\ & 1
   		\end{matrix}
\end{array} \right), \eta_{10}(\rho_1) =\left( \begin{array}{@{}c@{}}
 \begin{matrix}
   		1\ & 0\ & 0\\
   		0\ & 1\ & 0\\
   		0\ & 0\ & 1
   		\end{matrix}
\end{array} \right),$$ 
$$\eta_{10}(\tau_1) =\left( \begin{array}{@{}c@{}}
 \begin{matrix}
   		f\ & 0\ & 0\\
   		h\ & \dfrac{cf-ah+dh}{c}\ & 0\\
   		0\ & 0\ & 1
   		\end{matrix}
\end{array} \right) \text{ and \ } \eta_{10}(\gamma_1)=\eta_{10}(\gamma_2)=I_3,$$ 
where $a,c,d,f,h \in \mathbb{C}, ad\neq 0, c\neq 0$.
\vspace*{0.1cm}
\item[(11)] $\eta_{11}: STVB_2 \longrightarrow \mathrm{M}_3(\mathbb{C})$ such that
$$\eta_{11}(\sigma_1) =\left( \begin{array}{@{}c@{}}
  \begin{matrix}
   		a\ & 0\ & 0\\
   		0\ & d\ & 0\\
   		0\ & 0\ & 1
   		\end{matrix}
\end{array} \right), \eta_{11}(\rho_1) =\left( \begin{array}{@{}c@{}}
 \begin{matrix}
   		1\ & 0\ & 0\\
   		0\ & -1\ & 0\\
   		0\ & 0\ & 1
   		\end{matrix}
\end{array} \right),$$
$$\eta_{11}(\tau_1) =\left( \begin{array}{@{}c@{}}
 \begin{matrix}
   		f\ & 0\ & 0\\
   		0\ & k\ & 0\\
   		0\ & 0\ & 1
   		\end{matrix}
\end{array} \right) \text{ and \ } \eta_{11}(\gamma_1)=\eta_{11}(\gamma_2)=I_3,$$
where $a,d,f,k\in \mathbb{C}, ad\neq 0$.
\vspace*{0.1cm}
\item[(12)] $\eta_{12}: STVB_2 \longrightarrow \mathrm{M}_3(\mathbb{C})$ such that
$$\eta_{12}(\sigma_1) =\left( \begin{array}{@{}c@{}}
  \begin{matrix}
   		a\ & 0\ & 0\\
   		0\ & d\ & 0\\
   		0\ & 0\ & 1
   		\end{matrix}
\end{array} \right), \eta_{12}(\rho_1) =\left( \begin{array}{@{}c@{}}
 \begin{matrix}
   		-1\ & 0\ & 0\\
   		0\ & 1\ & 0\\
   		0\ & 0\ & 1
   		\end{matrix}
\end{array} \right),$$ 
$$\eta_{12}(\tau_1) =\left( \begin{array}{@{}c@{}}
 \begin{matrix}
   		f\ & 0\ & 0\\
   		0\ & k\ & 0\\
   		0\ & 0\ & 1
   		\end{matrix}
\end{array} \right) \text{ and \ } \eta_{12}(\gamma_1)=\eta_{12}(\gamma_2)=I_3,$$
where $a,d,f,k\in \mathbb{C}, ad\neq 0$.
\vspace*{0.1cm}
\item[(13)] $\eta_{13}: STVB_2 \longrightarrow \mathrm{M}_3(\mathbb{C})$ such that
$$\eta_{13}(\sigma_1) =\left( \begin{array}{@{}c@{}}
  \begin{matrix}
   		a\ & 0\ & 0\\
   		0\ & d\ & 0\\
   		0\ & 0\ & 1
   		\end{matrix}
\end{array} \right), \eta_{13}(\rho_1) =\left( \begin{array}{@{}c@{}}
 \begin{matrix}
   		1\ & 0\ & 0\\
   		0\ & 1\ & 0\\
   		0\ & 0\ & 1
   		\end{matrix}
\end{array} \right),$$
$$\eta_{13}(\tau_1) =\left( \begin{array}{@{}c@{}}
 \begin{matrix}
   		f\ & 0\ & 0\\
   		0\ & k\ & 0\\
   		0\ & 0\ & 1
   		\end{matrix}
\end{array} \right) \text{ and \ } \eta_{13}(\gamma_1)=\eta_{13}(\gamma_2)=I_3,$$
where $a,d,f,k\in \mathbb{C}, ad\neq 0$.
\end{itemize}
In addition, if $\eta(\tau_1)$ is invertible, then $\eta$ becomes a representation of $STVG_2$.
\end{theorem}

\begin{proof}
Since $\eta$ is a homogeneous local representation of $STVB_2$, we may set
$$\eta(\sigma_1) =\left( \begin{array}{@{}c@{}}
  \begin{matrix}
   		a\ & b\ & 0\\
   		c\ & d\ & 0\\
   		0\ & 0\ & 1
   		\end{matrix}
\end{array} \right),\ \eta(\rho_1) =\left( \begin{array}{@{}c@{}}
 \begin{matrix}
   		w\ & x\ & 0\\
   		y\ & z\ & 0\\
   		0\ & 0\ & 1
   		\end{matrix}
\end{array} \right),\ \eta(\tau_1) =\left( \begin{array}{@{}c@{}}
 \begin{matrix}
   		f\ & g\ & 0\\
   		h\ & k\ & 0\\
   		0\ & 0\ & 1
   		\end{matrix}
\end{array} \right),$$
$$\eta(\gamma_1) =\left( \begin{array}{@{}c@{}}
  \begin{matrix}
   		p\ & q\ & 0\\
   		r\ & s\ & 0\\
   		0\ & 0\ & 1
   		\end{matrix}
\end{array} \right), \text{ and }\ \eta(\gamma_2) =\left( \begin{array}{@{}c@{}}
 \begin{matrix}
  		1\ & 0\ & 0\\
   		0\ & p\ & q\\
   		0\ & r\ & s
   		\end{matrix}
\end{array} \right),$$
where $a,b,c,d,w,x,y,z,f,g,h,k,p,q,r,s \in \mathbb{C}$ satisfying the invertibility of matrices $\eta(\sigma_1), \eta(\rho_1), \eta(\gamma_1),$ and $\eta(\gamma_2)$. Again, according to the locality and homogeneity of the representation, it suffices to consider the following relations of the generators of $STVB_2$, and all other relations imply similar equations
$$\rho_1^2=1,$$
$$\gamma_1^2=\gamma_2^2=1,$$
$$\gamma_1\gamma_2=\gamma_2\gamma_1,$$
$$\rho_1\gamma_1=\gamma_2\rho_1,$$
$$\rho_1\sigma_1\rho_1=\gamma_2\gamma_1\sigma_1\gamma_1\gamma_2,$$
$$\sigma_1\tau_1=\tau_1\sigma_1,$$
$$\rho_1\tau_1\rho_1=\gamma_2\gamma_1\tau_1\gamma_1\gamma_2.$$
Applying these relations to the images of the generators of $STVB_2$ under the representation $\eta$ leads to a system of equations with sixteen unknowns. A similar proof as the proof of Theorem \ref{localTVB2} gives the desired result.
\end{proof}

Now, we study the faithfulness of the representations $\eta_i, 1\leq i \leq 13,$ given in Theorem \ref{localSTVB2}.

\begin{theorem}
Every representation of type $\eta_i, 1\leq i \leq 13,$ is unfaithful.
\end{theorem}
\begin{proof}
$\eta_1$ is not faithful as $\sigma_1\rho_1 \neq \rho_1\sigma_1$ in $STVB_2$ but their images under $\eta_1$ are the same, and $\eta_i, 2\leq i \leq 13,$ are unfaithful since $\gamma_1$ and $\gamma_2$ map to the identity matrix.
\end{proof}

We now study the irreducibility of the representations $\eta_i, 1\leq i \leq 13,$ given in Theorem \ref{localSTVB2}.

\begin{theorem}
All the representations in Theorem 4.1 are reducible to a two-dimensional representation and the following statements hold true.
\begin{enumerate}
    \item[(i)] The representation $\eta_1$ is not further reducible to a one-dimensional representation.
    \item[(ii)] The representation $\eta_2$ is further reducible to a one-dimensional representation. Moreover, $\eta_2$ is equivalent to a direct sum of three one-dimensional representations.
    \item[(iii)] The representations $\eta_i$, $3\leq i \leq 13$, are further reducible to a one-dimensional representation.
\end{enumerate}
\end{theorem}
\begin{proof}
As in the proof of Theorem 3.3, in each representation matrix the third standard basis vector $e_3$ is preserved, and the subspace spanned by $e_1$ and $e_2$ is invariant. Hence, each $\eta_i$ is reducible to a two-dimensional representation. To determine whether this $2\times 2$ subrepresentation is further reducible, we examine whether all the $2\times 2$ composition factors of $\eta_i(\sigma_1)$, $\eta_i(\rho_1)$, $\eta_i(\tau_1)$, $\eta_i(\gamma_1)$, and $\eta_i(\gamma_2)$ share a common eigenvector.\vspace{0.1cm}

Case 1: The representations of type $\eta_1$. The matrices $\eta_1(\rho_1)$, $\eta_1(\gamma_1)$, and $\eta_1(\gamma_2)$ coincide with $\zeta_1(\rho_1)$, $\zeta_1(\gamma_1)$, and $\zeta_1(\gamma_2)$, respectively. By Case 1 of the proof of Theorem 3.3, $\eta_1(\gamma_1)$ and $\eta_1(\rho_1)$ have no common eigenvector. Hence, $\eta_1$ is not further reducible to a one-dimensional representation.\vspace{0.1cm}

Case 2: The representations of type $\eta_2$. The $2\times 2$ composition factors of $\eta_2(\gamma_1)$ and $\eta_2(\gamma_2)$ are $I_2$. The $2\times 2$ composition factor of $\eta_2(\rho_1)$ has trace $0$ and determinant $-z^2-(1-z^2)=-1$, so its eigenvalues are $1$ and $-1$, with eigenvectors
$
u_1=\begin{pmatrix}
x \\
1+z
\end{pmatrix}
$
and
$
u_2=\begin{pmatrix}
x \\
z-1
\end{pmatrix},
$
respectively. Direct computations show that
\begin{gather*}
\begin{pmatrix}
a & b \\
\dfrac{b-bz^2}{x^2} & \dfrac{ax+2bz}{x}
\end{pmatrix}
u_1=\left(a+\dfrac{b(1+z)}{x}\right)u_1
\quad \text{and} \\
\begin{pmatrix}
a & b \\
\dfrac{b-bz^2}{x^2} & \dfrac{ax+2bz}{x}
\end{pmatrix}
u_2=\left(a-\dfrac{b(1-z)}{x}\right)u_2.
\end{gather*}
Since $\eta_2(\tau_1)$ is obtained from $\eta_2(\sigma_1)$ by replacing $a$ and $b$ with $f$ and $g$, respectively, the $2\times 2$ composition factor of $\eta_2(\tau_1)$ acts on $u_1$ by $f+\dfrac{g(1+z)}{x}$ and on $u_2$ by $f-\dfrac{g(1-z)}{x}$. Thus, $u_1$ and $u_2$ are common eigenvectors of all the $2\times 2$ composition factors, and so $\eta_2$ is further reducible to a one-dimensional representation. Moreover, $\det(u_1\ u_2)=-2x\neq 0$, so, setting
$
P=\begin{pmatrix}
x & x & 0\\
1+z & z-1 & 0\\
0 & 0 & 1
\end{pmatrix},
$
the matrices $P^{-1}\eta_2(\sigma_1)P$, $P^{-1}\eta_2(\rho_1)P$, $P^{-1}\eta_2(\tau_1)P$, $P^{-1}\eta_2(\gamma_1)P$, and $P^{-1}\eta_2(\gamma_2)P$ are all diagonal. Hence, $\eta_2$ is equivalent to a direct sum of three one-dimensional representations.\vspace{0.1cm}

Case 3: The representations of type $\eta_3$. The $2\times 2$ composition factor of $\eta_3(\rho_1)$ is $-I_2$, and those of $\eta_3(\gamma_1)$ and $\eta_3(\gamma_2)$ are $I_2$, so it suffices to consider $\eta_3(\sigma_1)$ and $\eta_3(\tau_1)$. Denote their $2\times 2$ composition factors by $A$ and $B$, respectively. From the relation $\sigma_1\tau_1=\tau_1\sigma_1$ in $STVB_2$, we have $AB=BA$. Let $\lambda$ be an eigenvalue of $A$ and let $E_\lambda$ be the corresponding eigenspace. For $v\in E_\lambda$, we have $A(Bv)=B(Av)=\lambda Bv$, so $E_\lambda$ is invariant under $B$. Since we are working over $\mathbb{C}$, the restriction of $B$ to $E_\lambda$ has an eigenvector $v$, which is then a common eigenvector of $A$ and $B$. Thus, $v$ is a common eigenvector of all the $2\times 2$ composition factors, and so $\eta_3$ is further reducible to a one-dimensional representation.\vspace{0.1cm}

Case 4: The representations of type $\eta_4$. This case is similar to Case 3 and follows the same reasoning, since the $2\times 2$ composition factor of $\eta_4(\rho_1)$ is $I_2$.\vspace{0.1cm}

Case 5: The representations of type $\eta_i$, $i=5,9$. The $2\times 2$ composition factors of $\eta_i(\sigma_1)$, $\eta_i(\rho_1)$, $\eta_i(\gamma_1)$, and $\eta_i(\gamma_2)$ are $aI_2$, $\mp I_2$, $I_2$, and $I_2$, respectively, so it suffices to consider $\eta_i(\tau_1)$. Since we are working over $\mathbb{C}$, the $2\times 2$ composition factor
$\begin{pmatrix}
f & g \\
h & k
\end{pmatrix}$
of $\eta_i(\tau_1)$ has an eigenvector $v$, which is then a common eigenvector of all the $2\times 2$ composition factors. Thus, $\eta_i$, $i=5,9$, is further reducible to a one-dimensional representation for all values of $f,g,h,k$.\vspace{0.1cm}

Case 6: The representations of type $\eta_i$, $i=6,7,8,10$. All $2\times 2$ composition factors of the matrices $\eta_i(\sigma_1)$, $\eta_i(\rho_1)$, $\eta_i(\tau_1)$, $\eta_i(\gamma_1)$, and $\eta_i(\gamma_2)$ are lower triangular, so they share a common eigenvector
$
\begin{pmatrix}
0 \\
1
\end{pmatrix}.
$
Thus, $\eta_i$ is reducible to a one-dimensional representation.\vspace{0.1cm}

Case 7: The representations of type $\eta_i$, $i=11,12,13$. All $2\times 2$ composition factors are diagonal, so they share the common eigenvectors
$\begin{pmatrix}
1 \\
0
\end{pmatrix}$
and
$\begin{pmatrix}
0 \\
1
\end{pmatrix}$.
Thus, $\eta_i$ is reducible to a one-dimensional representation; in fact, it is equivalent to a direct sum of three one-dimensional representations.
\end{proof}

In this section, we determined the forms of all complex homogeneous local representations of $STVB_2$ and $STVG_2$, which leads us to ask the following question, for $n\geq 3$, for future work.

\begin{question}
Consider $n\geq 3$ and let $\eta': STVB_n \longrightarrow \mathrm{M}_{n+1}(\mathbb{C})$ be a complex homogeneous local representation of $STVB_n$. What are the possible forms of $\eta'$?
\end{question}

\section{$\Phi$-Type Extensions of Complex Local Representations of $TVB_2$ to $STVB_2$}

In \cite{Nas20241}, Nasser studied the relation between local extensions and $\Phi$-type extensions to $SM_n$ of some local representations of $B_n$. In this section, we take a step toward answering the following question: if $\phi$ is a complex homogeneous local representation of $TVB_n$, $n\geq 2$, how are the local extensions and the $\Phi$-type extensions of $\phi$ to $STVB_n$ related?

We consider $n=2$ and we answer the mentioned question for the most important complex local representation of $TVB_2$, which is $\zeta_1$, obtained in Theorem \ref{localTVB2}.

Recall that every complex local extension of $\zeta_1$ to $STVB_2$ has the form of the representation $\eta_1$, obtained in Theorem \ref{localSTVB2}.

Now, we find the form of all $\Phi$-type extensions of $\zeta_1$ to $STVB_2$.

\begin{theorem} \label{FiType}
Let $\zeta_1: TVB_2 \longrightarrow \mathrm{GL}_3(\mathbb{C})$ be the complex local representation of $TVB_2$ obtained in Theorem \ref{localTVB2}. Let $t,u,v \in \mathbb{C}$ and let $\Phi_{t,u,v}: STVB_2\longrightarrow \mathrm{M}_3(\mathbb{C})$ be a $\Phi$-type extension of $\zeta_1$. Then, $\Phi_{t,u,v}$ acts on the generators of $STVB_2$ as follows.
$$\Phi_{t,u,v}(\sigma_1)=\zeta_1(\sigma_1), \Phi_{t,u,v}(\rho_1)=\zeta_1(\rho_1), \Phi_{t,u,v}(\gamma_1)=\zeta_1(\gamma_1), \Phi_{t,u,v}(\gamma_2)=\zeta_1(\gamma_2) \text{ and}$$
$$\Phi_{t,u,v}(\tau_1)=\left(
\begin{array}{ccc}
d \left(\dfrac{u x^2}{d^2 x^2-b^2}+t\right)+v & b \left(\dfrac{u x^2}{b^2-d^2 x^2}+t\right) & 0 \\
b \left(\dfrac{u}{b^2-d^2 x^2}+\dfrac{t}{x^2}\right) & d \left(\dfrac{u x^2}{d^2 x^2- b^2}+t\right)+v & 0 \\
0 & 0 & t+u+v\\
\end{array} \right),$$ 
where $b,d,x \in \mathbb{C}, d^2x^2-b^2\neq 0, x\neq 0$.
\end{theorem}

\begin{proof}
After computing $\zeta_1(\sigma_1^{-1})$, we can directly see that $\Phi_{t,u,v}(\tau_1)=t\zeta_1(\sigma_1)+u\zeta_1(\sigma_1^{-1})+vI_3,$ which implies that
$$\Phi_{t,u,v}(\tau_1)=t\left( \begin{array}{@{}c@{}}
  \begin{matrix}
   		d & b & 0\\
   		\dfrac{b}{x^2} & d & 0\\
   		0 & 0 & 1
   		\end{matrix}
\end{array} \right)+u\left(
\begin{array}{ccc}
 \dfrac{d}{d^2-\dfrac{b^2}{x^2}} & -\dfrac{b}{d^2-\dfrac{b^2}{x^2}} & 0 \\
 -\dfrac{b}{x^2 \left(d^2-\dfrac{b^2}{x^2}\right)} & \dfrac{d}{d^2-\dfrac{b^2}{x^2}} & 0 \\
 0 & 0 & 1 \\
\end{array}
\right)+vI_3$$
$\hspace*{2.05cm}\ \ \ \ =\left(
\begin{array}{ccc}
d \left(\dfrac{u x^2}{d^2 x^2-b^2}+t\right)+v & b \left(\dfrac{u x^2}{b^2-d^2 x^2}+t\right) & 0 \\
b \left(\dfrac{u}{b^2-d^2 x^2}+\dfrac{t}{x^2}\right) & d \left(\dfrac{u x^2}{d^2 x^2- b^2}+t\right)+v & 0 \\
0 & 0 & t+u+v\\
\end{array} \right).$

\end{proof}

As a consequence, the following corollary is to compare complex local extensions and $\Phi$-type extensions of the local representation $\zeta_1$ of $TVB_2$ to $STVB_2$.

\begin{corollary}
Let $\zeta_1: TVB_2 \longrightarrow \mathrm{GL}_3(\mathbb{C})$ be the complex local representation of $TVB_2$ obtained in Theorem \ref{localTVB2}. Consider the complex local extension, $\eta_1$, of $\zeta_1$, obtained in Theorem \ref{localSTVB2}, and the $\Phi$-type extension, $\Phi_{t,u,v}$, of $\zeta_1$, obtained in Theorem \ref{FiType}. Then
$$\eta_1(\tau_1)=\Phi_{t,u,v}(\tau_1)$$
if and only if
$$t=\dfrac{b^2 g+x^2 (b (f-1)-(d-1) d g)}{b^3-b (d-1)^2 x^2},$$
$$u=-\dfrac{(b-d x) (b+d x) (b (f-1)-d g+g)}{b^3-b (d-1)^2 x^2}$$
and
$$v=\dfrac{b^3 f-b^2 d g-b x^2 (d (d f-2)+f)+d \left(d^2-1\right) g x^2}{b^3-b (d-1)^2 x^2}.$$
\end{corollary} 
 
This section shows that not every complex local extension of a complex local representation of $TVB_2$ to $STVB_2$ is of $\Phi$-type. This raises the following open questions.

\begin{question}
For $n\geq 2$, let $\phi$ be a representation of $TVB_n$. What is the precise relationship between the local extensions and the $\Phi$-type extensions of $\phi$ to $STVB_n$?
\end{question}

\begin{question}
For $n\geq 2$, is there a general form for extending a representation $\phi$ of $TVB_n$ to $STVB_n$, other than the complex local and $\Phi$-type extensions?
\end{question}
 
\section{Conclusion}

This study suggests several directions for future research, particularly on extending known braid group representations to the singular twisted virtual braid groups. Building on Bardakov's work on extensions of the Burau representation to $SM_n$ \cite{Bar2024}, a natural next step is to study analogous extensions to $STVB_n$ for all $n \geq 2$. This requires defining images of the generators $\rho$ and $\gamma$ that preserve the relations of $STVB_n$. One can then ask whether every such extension is of $\Phi$-type, as in Bardakov's results.

Similarly, motivated by Bardakov's classification of extensions of the Lawrence--Krammer--Bigelow representation (LKBR) of $B_n$ to $SM_n$ for $n=3,4$, one can study the corresponding extensions to $STVB_n$, again by suitably defining the images of $\rho$ and $\gamma$. The question is then whether all extensions of the LKBR to $STVB_n$ for $n=3,4$ are also of $\Phi$-type. Answering these questions would deepen our understanding of the representation theory of singular twisted virtual braid groups.

\section{Acknowledgment}
The second author acknowledges the support of the University Grants Commission (UGC), India, for a research fellowship with NTA Ref. No.231610035955 . The third author acknowledges the support of the Anusandhan National Research Foundation (ANRF) with sanction order no. CRG/2023/004921, and the National
Board of Higher Mathematics (NBHM), Government of India, under grant-in-aid with F.No. 02011/2/20223 NBHM(R.P.)/ R\&D II/970.


\end{document}